\UseRawInputEncoding
\documentclass[a4paper]{article}

\usepackage{yfonts}
\usepackage{marvosym}
\usepackage[numbers]{natbib}

\usepackage{amsmath,amssymb,tikz} 
\usepackage{pict2e}

\usepackage{todonotes}
\usetikzlibrary{decorations.markings}

\usetikzlibrary{calc,patterns,angles,quotes}
\usetikzlibrary{arrows,calc,chains, positioning, shapes.geometric,shapes.symbols,decorations.markings,arrows.meta}

\usepackage[colorlinks=true, pdfstartview=FitV, linkcolor=blue, citecolor=blue, urlcolor=blue]{hyperref}

\numberwithin{equation}{section}

\def\beq{\begin{equation}}
\def\eeq{\end{equation}}

\def\bit{\begin{itemize}}
	\def\eit{\end{itemize}}

\makeatletter
\def\eqalign#1{\null\vcenter{\def\\{\cr}\openup\jot\m@th
		\ialign{\strut$\displaystyle{##}$\hfil&$\displaystyle{{}##}$\hfil
			\crcr#1\crcr}}\,}
\makeatother

\newcommand{\re}{\text{\upshape Re\,}}
\newcommand{\im}{\text{\upshape Im\,}}

\newcommand{\PP}{{\mathbb P}}
\newcommand{\R}{{\mathbb R}}
\newcommand{\C}{{\mathbb C}}

\newcommand{\E}{{\mathbb E}}
\newcommand{\N}{{\mathbb N}}

\newcommand{\T}{{\mathbb T}}

\newcommand{\ds}{\displaystyle}
\def\bigO{{\cal O}}
\newenvironment{proof}%
{\rm \trivlist \item[\hskip \labelsep{\bf Proof. }]}%
{\hspace*{\fill}$\Box$\endtrivlist}

\def\Xint#1{\mathchoice
   {\XXint\displaystyle\textstyle{#1}}%
   {\XXint\textstyle\scriptstyle{#1}}%
   {\XXint\scriptstyle\scriptscriptstyle{#1}}%
   {\XXint\scriptscriptstyle\scriptscriptstyle{#1}}%
   \!\int}
\def\XXint#1#2#3{{\setbox0=\hbox{$#1{#2#3}{\int}$}
     \vcenter{\hbox{$#2#3$}}\kern-.5\wd0}}
\def\ddashint{\Xint=}
\def\dashint{\Xint-}

\newcommand{\vertiii}[1]{{\left\vert\kern-0.25ex\left\vert\kern-0.25ex\left\vert #1 
    \right\vert\kern-0.25ex\right\vert\kern-0.25ex\right\vert}}

\begin{document}
	\tikzset{middlearrow/.style={
			decoration={markings,
				mark= at position 0.6 with {\arrow{#1}} ,
			},
			postaction={decorate}
		}
	}
	
\newtheorem{theorem}{Theorem}
\newtheorem{acknowledgement}[theorem]{Acknowledgement}
\newtheorem{remark}[theorem]{Remark}
\newtheorem{lemma}[theorem]{Lemma}
\newtheorem{proposition}[theorem]{Proposition}
\newtheorem{corollary}[theorem]{Corollary}
\numberwithin{equation}{section}
\numberwithin{theorem}{section}
	
\def\Xint#1{\mathchoice
	{\XXint\displaystyle\textstyle{#1}}%
	{\XXint\textstyle\scriptstyle{#1}}%
	{\XXint\scriptstyle\scriptscriptstyle{#1}}%
	{\XXint\scriptscriptstyle\scriptscriptstyle{#1}}%
	\!\int}
\def\XXint#1#2#3{{\setbox0=\hbox{$#1{#2#3}{\int}$ }
	\vcenter{\hbox{$#2#3$ }}\kern-.59\wd0}}
\def\ddashint{\Xint=}
\def\dashint{\Xint-}

\let\oldbibliography\thebibliography
\renewcommand{\thebibliography}[1]{\oldbibliography{#1}
\setlength{\itemsep}{-0.2pt}}

\allowdisplaybreaks
	
\title{Toeplitz determinants with a one-cut regular potential \\  and Fisher--Hartwig singularities \\ I. Equilibrium measure supported on the unit circle}
\author{Elliot Blackstone\footnote{Department of Mathematics, University of Michigan, Ann Arbor, USA. E-mail: eblackst@umich.edu}, Christophe Charlier\footnote{Centre for Mathematical Sciences, Lund University, 22100 Lund, Sweden. E-mail: 	christophe.charlier@math.lu.se}, Jonatan Lenells\footnote{Department of Mathematics, KTH Royal Institute of Technology, 100 44 Stockholm, Sweden. E-mail: jlenells@kth.se}}
	
\maketitle
	
\tikzset{->-/.style={decoration={
				markings,
				mark=at position #1 with {\arrow{latex}}},postaction={decorate}}}
	
\tikzset{-<-/.style={decoration={
				markings,
				mark=at position #1 with {\arrowreversed{latex}}},postaction={decorate}}}

\begin{abstract}
We consider Toeplitz determinants whose symbol has: (i) a one-cut regular potential $V$, (ii) Fisher--Hartwig singularities, and (iii) a smooth function in the background. The potential $V$ is associated with an equilibrium measure that is assumed to be supported on the whole unit circle. For constant potentials $V$, the equilibrium measure is the uniform measure on the unit circle and our formulas reduce to well-known results for Toeplitz determinants with Fisher--Hartwig singularities. For non-constant $V$, our results appear to be new even in the case of no Fisher--Hartwig singularities. As applications of our results, we derive various statistical properties of a determinantal point process which generalizes the circular unitary ensemble.
\end{abstract}



\section{Introduction}
In this work, we obtain large $n$ asymptotics of the Toeplitz determinant 
\begin{align}
D_n(\vec\alpha,\vec\beta,V,W)& := \det (f_{j-k})_{j,k=0,\dots,n-1}, \qquad f_k:=\frac{1}{2\pi}\int_0^{2\pi}f(e^{i\theta})e^{-ik\theta}d\theta,
 \label{uni: toeplitz det}
\end{align}
where $f$ is supported on the unit circle $\mathbb{T}=\{z\in\mathbb{C}:|z|=1\}$ and is of the form
\begin{align}\label{uni: f}
    f(z)=e^{-nV(z)}e^{W(z)}\omega(z), \qquad z\in \mathbb{T}.
\end{align}
We assume that $V$ and $W$ are analytic in a neighborhood of $\mathbb{T}$ and that the potential $V$ is real-valued on $\mathbb{T}$. The function $\omega(z)=\omega(z;\vec\alpha,\vec\beta)$ in \eqref{uni: f} contains Fisher--Hartwig singularities and is defined in \eqref{uni: omega} below. 
Since the functions $V$ and $W$ are analytic on $\mathbb{T}$, there exists an open annulus $U$ containing $\mathbb{T}$ on which they admit Laurent series representations of the form
\begin{align}
& V(z) = V_{0} + V_{+}(z) + V_{-}(z), & & V_{+}(z) = \sum_{k=1}^{+\infty} V_{k}z^{k}, & & V_{-}(z) = \sum_{k=-\infty}^{-1} V_{k}z^{k},  \label{V as a Laurent series} \\
& W(z) = W_{0} + W_{+}(z) + W_{-}(z), & & W_{+}(z) = \sum_{k=1}^{+\infty} W_{k}z^{k}, & & W_{-}(z) = \sum_{k=-\infty}^{-1} W_{k}z^{k}, \label{W as a Laurent series}
\end{align}
where $V_{k}, W_{k}\in \mathbb{C}$ are the Fourier coefficients of $V$ and $W$, i.e. $V_{k} = \frac{1}{2\pi}\int_0^{2\pi}V(e^{i\theta})e^{-ik\theta}d\theta$ and similarly for $W_{k}$. 
Associated to $V$ there is an equilibrium measure $\mu_{V}$, which is the unique minimizer of the functional
\begin{equation}\label{functional}
\mu \mapsto \iint \log \frac{1}{|z-s|} d\mu(z)d\mu(s) + \int V(z)d\mu(z)
\end{equation}
among all Borel probability measures $\mu$ on $\mathbb{T}$. In this paper we make the assumption that $\mu$ is supported on the whole unit circle. We further assume that $V$ is regular, i.e. that the function $\psi$ given by
\begin{align}\label{psi solved explicitly}
\psi(z) = \frac{1}{2\pi} - \frac{1}{2\pi} \sum_{\ell = 1}^{+\infty} \ell ( V_{\ell}z^{\ell} + \overline{V_{\ell}}z^{-\ell}), \qquad z \in U,
\end{align}
is strictly positive on $\mathbb{T}$. Under these assumptions, we show in Appendix \ref{section:equilibrium measure} that
\begin{equation}\label{uni: dmu}
d\mu_V(e^{i\theta})=\psi(e^{i\theta})d\theta, \qquad \theta\in [0,2\pi).
\end{equation}

The function $\omega$ appearing in \eqref{uni: f} is defined by
\begin{equation}\label{uni: omega}
\omega(z) = \prod_{k=0}^{m} \omega_{\alpha_{k}}(z)\omega_{\beta_{k}}(z),
\end{equation}
where $\omega_{\alpha_{k}}(z)$ and $\omega_{\beta_{k}}(z)$ are defined for $z=e^{i\theta}$ by
\begin{equation}\label{uni: FH pieces}
\omega_{\alpha_{k}}(z) = |z-t_k|^{\alpha_{k}}, \quad \omega_{\beta_{k}}(z) = e^{i(\theta -\theta_{k})\beta_{k}} \times \left\{ \begin{array}{l l}
e^{i\pi\beta_{k}}, & \mbox{ if } 0 \leq \theta < \theta_{k}, \\
e^{-i \pi \beta_{k}}, & \mbox{ if } \theta_{k} \leq \theta < 2\pi,
\end{array}  \right. \quad \theta \in [0,2\pi),
\end{equation}
and   
\begin{equation}\label{uni: theta k}
t_k:=e^{i\theta_k}, \qquad 0 =\theta_{0} < \theta_{1} < \cdots < \theta_{m} < 2\pi.
\end{equation}
At $t_k=e^{i\theta_{k}}$, the functions $\omega_{\alpha_{k}}$ and $\omega_{\beta_{k}}$ have root- and jump-type singularities, respectively. Note that $\omega_{\beta_{k}}$ is continuous at $z=1$ if $k \neq 0$. 
We allow the parameters $\theta_{1},\dots,\theta_{m}$ to vary with $n$, but we require them to lie in a compact subset of $(0,2\pi)_{\mathrm{ord}}^{m}:=\{(\theta_{1},\dots,\theta_{m}): 0 < \theta_{1} < \cdots < \theta_{m} < 2\pi\}$. 

To summarize, the $n \times n$ Toeplitz determinant \eqref{uni: toeplitz det} depends on $n$, $m$, $V$, $W$, $\vec{t} = (t_{1},\dots,t_{m})$, $\vec{\alpha}=(\alpha_1,\dots,\alpha_m)$ and $\vec{\beta} = (\beta_{1},\dots,\beta_{m})$, but for convenience the dependence on $m$ and $\vec{t}$ is omitted in the notation $D_n(\vec\alpha,\vec\beta,V,W)$. We now state our main result.

\begin{theorem}[Large $n$ asymptotics of $D_{n}(\vec{\alpha},\vec{\beta},V,W)$]\label{theorem U} 
Let $m \in \mathbb{N} :=\{0,1,\dots\}$, and let $t_{k}=e^{i\theta_{k}}$, $\alpha_{k}\in \mathbb{C}$ and $\beta_{k} \in \mathbb{C}$ be such that
\begin{equation*}
0 = \theta_{0} < \theta_{1} < \dots < \theta_{m} < 2\pi, \quad \mbox{ and } \quad \re \alpha_{k} > -1, \quad \re \beta_{k} \in (-\tfrac{1}{2},\tfrac{1}{2}) \quad \mbox{ for } k=0,\dots,m.
\end{equation*}  
Let $V: \mathbb{T}\to\mathbb{R}$ and $W: \mathbb{T}\to \mathbb{C}$, and suppose $V$ and $W$ can be extended to analytic functions in a neighborhood of $\mathbb{T}$. Suppose that the equilibrium measure $d\mu_V(e^{i\theta})=\psi(e^{i\theta})d\theta$ associated to $V$ is supported on $\mathbb{T}$ and that $\psi > 0$ on $\mathbb{T}$. Then, as $n \to \infty$, 
\begin{equation}\label{asymp thm}
D_{n}(\vec{\alpha},\vec{\beta},V,W) = \exp\left(C_{1} n^{2} + C_{2} n + C_{3} \log n + C_{4} + \bigO \Big( n^{-1+2\beta_{\max}} \Big)\right),
\end{equation}
with $\beta_{\max} = \max \{ |\re \beta_{1}|,\dots,|\re \beta_{m}| \}$ and 
\begin{align*}
 C_{1} = & -\frac{V_{0}}{2}-\frac{1}{2}\int_0^{2\pi}V(e^{i\theta}) d\mu_{V}(e^{i\theta}), \\
 C_{2} = &\; \sum_{k=0}^{m} \frac{\alpha_{k}}{2}(V(t_{k})-V_{0}) - \sum_{k=0}^{m} 2i\beta_{k} \im(V_{+}(t_{k})) + \int_0^{2\pi}W(e^{i\theta})d\mu_V(e^{i\theta}), \\
 C_{3} = &\; \sum_{k=0}^{m} \bigg( \frac{\alpha_{k}^{2}}{4}-\beta_{k}^{2} \bigg), \\
 C_{4} = &\;\sum_{\ell = 1}^{+\infty} \ell W_{\ell}W_{-\ell} - \sum_{k=0}^{m} \frac{\alpha_{k}}{2}(W(t_{k})-W_{0}) + \sum_{k=0}^{m}\beta_{k} \big(W_{+}(t_{k})-W_{-}(t_{k})\big)  \\
& + \sum_{0 \leq j < k \leq m} \bigg\{ \frac{\alpha_{j} i \beta_{k} - \alpha_{k} i \beta_{j}}{2}(\theta_{k}-\theta_{j }-\pi) + \bigg( 2\beta_{j}\beta_{k}-\frac{\alpha_{j}\alpha_{k}}{2} \bigg) \log |t_{j}-t_{k}| \bigg\} \\
& + \sum_{k=0}^{m} \log \frac{G(1+\frac{\alpha_{k}}{2}+\beta_{k})G(1+\frac{\alpha_{k}}{2}-\beta_{k})}{G(1+\alpha_{k})} + \sum_{k=0}^m\frac{\beta_{k}^{2}-\frac{\alpha_{k}^{2}}{4}}{\psi(t_k)}\left(\frac{1}{2\pi}-\psi(t_k)\right),
\end{align*}
where $G$ is Barnes' $G$-function. Furthermore, the above asymptotics are uniform for all $\alpha_{k}$ in compact subsets of $\{z \in \mathbb{C}: \re z >-1\}$, for all $\beta_{k}$ in compact subsets of $\{z  \in \mathbb{C}: \re z \in (-\frac{1}{2},\frac{1}{2})\}$, and for all $(\theta_{1},\dots,\theta_{m})$ in compact subsets of $(0,2\pi)_{\mathrm{ord}}^{m}$. The above asymptotics can also be differentiated with respect to $\alpha_{0},\dots,\alpha_{m},\beta_{0},\dots,\beta_{m}$ as follows: if $k_{0},\dots,k_{2m+1}\in \mathbb{N}$, $k_{0}+\dots+k_{2m+1}\geq 1$ and $\partial^{\vec{k}}:=\partial_{\alpha_{0}}^{k_{0}}\dots\partial_{\alpha_{m}}^{k_{m}}\partial_{\beta_{0}}^{k_{m+1}}\dots\partial_{\beta_{m}}^{k_{2m+1}}$, then
\begin{align}\label{der of error in thm}
\partial^{\vec{k}}\Big( \log D_{n}(\vec{\alpha},\vec{\beta},V,W) - \log \widehat{D}_{n} \Big) = \bigO \bigg( \frac{(\log n)^{k_{m+1}+\dots+k_{2m+1}}}{n^{1-2\beta_{\max}}} \bigg), \qquad \mbox{as } n \to + \infty,
\end{align}
where $\widehat{D}_{n}$ denotes the right-hand side of \eqref{asymp thm} without the error term.
\end{theorem}

\subsection{History and related work}
In the case when the potential $V(z)$ in \eqref{uni: f} vanishes identically, the asymptotic evaluation of Toeplitz determinants of the form \eqref{uni: toeplitz det} has a long and distinguished history.
The first important result was obtained by Szeg\H{o} in 1915 who determined the leading behavior of $D_{n}(\vec{\alpha},\vec{\beta},V,W)$ in the case when $\vec{\alpha} = \vec{\beta} = \vec{0}$ and $V = 0$, that is, when the symbol $f(z)$ is given by $f(z) = e^{W(z)}$. In our notation, this result, known as the first Szeg\H{o} limit theorem \cite{S1915}, can be expressed as
\begin{align}\label{firstSzegotheorem}
D_{n}(\vec{0},\vec{0},0,W) = \exp\left(\frac{n}{2\pi}\int_0^{2\pi}W(e^{i\theta})d\theta + o(n)\right) \qquad \text{as $n \to \infty$}.
\end{align}
Later, in the 1940's, it became clear from the pioneering work of Kaufmann and Onsager that a more detailed understanding of the error term in (\ref{firstSzegotheorem}) could be used to compute two-point correlation functions in the two-dimensional Ising model in the thermodynamic limit \cite{KO1949}. This inspired Szeg\H{o} to seek for a stronger version of (\ref{firstSzegotheorem}). The outcome was the so-called strong Szeg\H{o} limit theorem \cite{S1952}, which in our notation states that 
\begin{align}\label{SSLT}
D_{n}(\vec{0},\vec{0},0,W) = \exp\left(\frac{n}{2\pi}\int_0^{2\pi}W(e^{i\theta})d\theta + \sum_{\ell = 1}^{+\infty} \ell W_{\ell}W_{-\ell} + o(1)\right) \qquad \text{as $n \to \infty$}.
\end{align}
We observe that if $V = 0$, then $d\mu_V(e^{i\theta}) = \frac{d\theta}{2\pi}$; thus Szeg\H{o}'s theorems are consistent with our main result, Theorem \ref{asymp thm}, in the special case when $\vec{\alpha} = \vec{\beta} = \vec{0}$ and $V = 0$. (The strong Szeg\H{o} theorem actually holds under much weaker assumptions on $W$ than what is assumed in this paper, see e.g. the survey \cite{BasorSzego}.)

In a groundbreaking paper from 1968, Fisher and Hartwig introduced a class of singular symbols $f(z)$ for which they convincingly conjectured a detailed asymptotic formula for the associated Toeplitz determinant \cite{FisherHartwig}. The Fisher--Hartwig class consists of symbols $f(z)$ of the form (\ref{uni: f}) with $V = 0$. In our notation, the Fisher--Hartwig conjecture can be formulated as
\begin{align}\label{FHconjecture}
D_{n}(\vec{\alpha},\vec{\beta},0,W) \sim \exp\left(\frac{n}{2\pi}\int_0^{2\pi}W(e^{i\theta})d\theta + \sum_{k=0}^{m} \bigg( \frac{\alpha_{k}^{2}}{4}-\beta_{k}^{2} \bigg)\log{n} + C_4\right) \qquad \text{as $n \to \infty$},
\end{align}
where $C_4$ is a constant to be determined, and the Fisher--Hartwig singularities are encoded in the vectors $\vec{\alpha}$ and $\vec{\beta}$.
Symbols with Fisher--Hartwig singularities arise in many applications. For example, in the 1960's, Lenard proved \cite{L1964} that no Bose-Einstein condensation exists in the ground state for a one-dimensional system of impenetrable bosons by considering Toeplitz determinants with symbols of the form $f(z) = |z-e^{i\theta_1}| |z - e^{-i\theta_1}|$ with $\theta_1 \in \R$. Lenard's proof hinges on an inequality whose proof was provided by Szeg\H{o}, see \cite[Theorem 2]{L1964}. We observe that (\ref{FHconjecture}) is consistent with Theorem \ref{asymp thm} in the special case when $V = 0$.
 
There are too many works devoted to proofs and generalizations of the Fisher--Hartwig conjecture (\ref{FHconjecture}) for us to cite them all, but we refer to \cite{Widom2, Basor, BS1986} for some early works, and to \cite{BT1991, BasMor, B1995, DIK2013} for four reviews.
The current state-of-the-art for non-merging singularities and for $\vec{\alpha}$, $\vec{\beta}$ in compact subsets was set by Ehrhardt in his 1997 Ph.D. thesis (see \cite{Ehr}) and by Deift, Its, and Krasovsky in \cite{DIK2011, DeiftItsKrasovsky}. Since our proof builds on the results for the case of $V = 0$, we have included a version of the asymptotic formulas of \cite{Ehr, DIK2011, DeiftItsKrasovsky} in Theorem \ref{thm:V=0}. 
We also refer to \cite{ClKr, Fahs} for studies of merging Fisher--Hartwig singularities with $V=0$,  and to \cite{ChCl2} for the case of large discontinuities with $V=0$.

Note that if $V=V_{0}$ is a constant, then $D_{n}(\vec{\alpha},\vec{\beta},V_{0},W)=e^{-n^{2}V_{0}}D_{n}(\vec{\alpha},\vec{\beta},0,W)$.

The novelty of the present work is that we consider symbols that include a non-constant potential $V$; we are not aware of any previous works on the unit circle including such potentials. Our main result is formulated under the assumption that $\re \beta_{k} \in (-\frac{1}{2},\frac{1}{2})$ for all $k$. The general case where $\re \beta_{k} \in \mathbb{R}$ was treated in the case of $V=0$ in \cite{DIK2011}. 
Asymptotic formulas for Hankel determinants with a one-cut regular potential $V$ and Fisher--Hartwig singularities were obtained in \cite{BerWebbWong, Charlier, ChGha}, and the corresponding multi-cut case was considered in \cite{CFWW2021}. Our proofs draw on some of the techniques developed in these papers.

\subsection{Application: A determinantal point process on the unit circle}
The Toeplitz determinant \eqref{uni: toeplitz det} admits the Heine representation
\begin{align}
D_n(\vec\alpha,\vec\beta,V,W)& = \frac{1}{n!(2\pi)^n}\int_{[0,2\pi]^{n}} \prod_{1 \leq j < k \leq n} |e^{i\phi_{k}}-e^{i\phi_{j}}|^{2}\prod_{j=1}^{n} f(e^{i\phi_{j}})d\phi_{j}.
\end{align}
This suggests that the results of Theorem \ref{asymp thm} can be applied to obtain information about the point process on $\mathbb{T}$ defined by the probability measure
\begin{align}\label{point process}
\frac{1}{n! (2\pi)^n Z_{n}} \prod_{1 \leq j < k \leq n} |e^{i\phi_{k}}-e^{i\phi_{j}}|^{2}\prod_{j=1}^{n} e^{-nV(e^{i\phi_{j}})}d\phi_{j}, \qquad \phi_{1},\dots,\phi_{n}\in[0,2\pi),
\end{align}
where $Z_{n} = D_{n}(\vec{0},\vec{0},V,0)$ is the normalization constant (also called the partition function). In what follows, we use Theorem \ref{asymp thm} to obtain smooth statistics, log statistics, counting statistics, and rigidity bounds for the point process (\ref{point process}).
In the case of constant $V$, the point process (\ref{point process}) describes the distribution of eigenvalues of matrices drawn from the circular unitary ensemble and has already been widely studied. We are not aware of any earlier work where the process \eqref{point process} is considered explicitly for non-constant $V$. However, the point process \eqref{point process}, but with $nV(e^{i\phi})$ replaced by the highly oscillatory potential $V(e^{in\phi})$, is studied in \cite{ForRods, BaikOscillatory}. We also refer to \cite{BS2021, BF2022} for other determinantal generalizations of the circular unitary ensemble.

Let $\mathsf{p}_{n}(z):=\prod_{j=1}^{n}(e^{i\phi_{j}}-z)$ be the characteristic polynomial associated to \eqref{point process}, and define $\log \mathsf{p}_{n}(z)$ for $z \in \mathbb{T}\setminus \{e^{i\phi_{1}},\dots,e^{i\phi_{n}}\}$ by
\begin{align*}
\log \mathsf{p}_{n}(z) := \sum_{j=1}^{n} \log (e^{i\phi_{j}}-z), \qquad \im \log (e^{i\phi_{j}}-z) := \frac{\phi_{j} + \arg_{0} z}{2} + \begin{cases}
\frac{3\pi}{2}, & \mbox{if } 0 \leq \phi_{j} < \arg_{0} z, \\
\frac{\pi}{2}, & \mbox{if } \arg_{0} z < \phi_{j} < 2\pi,
\end{cases}
\end{align*}
where $\arg_{0} z \in [0,2\pi)$. In particular, if $\theta_{k}\notin \{\phi_{1},\dots,\phi_{n}\}$,
\begin{align}\label{the jumps are always a mess}
e^{2i \beta_{k}(\im \log \mathsf{p}_{n}(t_{k})-n\theta_{k}-n\pi)} = \prod_{j=1}^{n} \omega_{\beta_{k}}(e^{i\phi_{j}}) = e^{-i\beta_{k}(\pi+\theta_{k}) n }e^{2\pi i \beta_{k}N_{n}(\theta_{k})}\prod_{j=1}^{n} e^{i\beta_{k}\phi_{j}},
\end{align}
where $N_{n}(\theta):=\#\{\phi_{j} \in [0,\theta ]\} \in \{0,1,\dots,n\}$. Using the first identity in \eqref{the jumps are always a mess} and the fact that $\{\theta_{0},\dots,\theta_{m}\} \cap \{\phi_{1},\dots,\phi_{n}\} = \emptyset$ with probability one, it is straightforward to see that
\begin{align}\label{lol2}
& \mathbb{E}\bigg[\prod_{j=1}^{n}e^{W(e^{i\phi_{j}})}\prod_{k=0}^{m}e^{\alpha_{k}\re \log \mathsf{p}_{n}(t_{k})}e^{2i\beta_{k}(\im \log \mathsf{p}_{n}(t_{k})-n\theta_{k}-n\pi)}\bigg] = \frac{D_{n}(\vec{\alpha},\vec{\beta},V,W)}{D_{n}(\vec{0},\vec{0},V,0)}.
\end{align}
Furthermore, if $\beta_{0}=-\beta_{1}-\dots-\beta_{m}$, then the second identity in \eqref{the jumps are always a mess} together with \eqref{lol2} implies
\begin{align}\label{MGF appli}
\frac{D_{n}(\vec{\alpha},\vec{\beta},V,W)}{D_{n}(\vec{0},\vec{0},V,0)} = \prod_{k=1}^{m} e^{-i \beta_{k} \theta_{k} n } \times \mathbb{E}\bigg[\prod_{j=1}^{n}e^{W(e^{i\phi_{j}})}\prod_{k=0}^{m}|\mathsf{p}_{n}(t_{k})|^{\alpha_{k}}e^{2 \pi i\beta_{k}N_{n}(\theta_{k})}\bigg].
\end{align}

\begin{lemma}\label{lemma:some id for alpha and beta}
For any $z \in \mathbb{T}$, we have
\begin{align}
& \frac{V(z)-V_{0}}{2} = \int_{0}^{2\pi} \log |e^{i\theta}-z|d\mu_{V}(e^{i\theta}), \label{id1 lemma} \\
& \frac{\arg_{0} z}{2\pi} - \frac{\im V_{+}(z) - \im V_{+}(1)}{\pi} = \int_{0}^{\arg_{0} z} d\mu_{V}(e^{i\theta}). \label{id2 lemma}
\end{align}
\end{lemma}
\begin{proof}
The equilibrium measure $\mu_{V}$ is uniquely characterized by the Euler-Lagrange variational equality
\begin{align}
2 \int_{0}^{2\pi} \log |z-e^{i\theta}| d\mu_{V}(e^{i\theta}) = V(z) - \ell, & & \mbox{ for } z \in \mathbb{T}, 
\label{var equality}
\end{align} 
where $\ell \in \mathbb{R}$ is a constant, see e.g. \cite{SaTo}.
In particular, the identity \eqref{id1 lemma} is equivalent to the statement that $\ell=V_{0}$. 
The equality $\ell=V_{0}$ can be established by integrating \eqref{var equality} over $z=e^{i\phi} \in \mathbb{T}$ and dividing by $2\pi$:
\begin{align*}
\ell = \int_{0}^{2\pi} \ell \frac{d\phi}{2\pi} = \int_{0}^{2\pi} \bigg( V(z) - 2 \int_{0}^{2\pi} \log|e^{i\phi}-e^{i\theta}|d\mu_{V}(e^{i\theta}) \bigg)\frac{d\phi}{2\pi} = V_{0},
\end{align*}
where we have used the well-known (see e.g. \cite[Example 0.5.7]{SaTo}) identity $\int_{0}^{2\pi} \log|e^{i\phi}-e^{i\theta}| \frac{d\phi}{2\pi} =0$ for $\theta \in [0,2\pi)$. This proves \eqref{id1 lemma}.
The identity \eqref{id2 lemma} follows from \eqref{psi solved explicitly} and \eqref{V as a Laurent series}.
\end{proof}

Combining \eqref{MGF appli}, Theorem \ref{theorem U}  and Lemma \ref{lemma:some id for alpha and beta}, we get the following.
\begin{theorem}\label{thm:MGF}
Let $m \in \mathbb{N}$, and let $t_{k}=e^{i\theta_{k}}$, $\alpha_{0},\dots,\alpha_{m}\in \mathbb{C}$ and $u_{1},\dots,u_{m} \in \mathbb{C}$ be such that
\begin{align*}
& 0 = \theta_{0} < \theta_{1} < \dots < \theta_{m} < 2\pi, \quad \mbox{ and } \quad \re \alpha_{k} > -1, \quad \im u_{k} \in (-\pi,\pi) \quad \mbox{for all } k.
\end{align*}  
Let $V: \mathbb{T}\to\mathbb{R}$, $W: \mathbb{T}\to \mathbb{C}$, and suppose $V$, $W$ can be extended to analytic functions in a neighborhood of $\mathbb{T}$. Suppose that the equilibrium measure $d\mu_V(e^{i\theta})=\psi(e^{i\theta})d\theta$ associated to $V$ is supported on $\mathbb{T}$ and that $\psi > 0$ on $\mathbb{T}$. Then, as $n \to \infty$, we have 
\begin{align}\label{exp in thm}
\mathbb{E}\bigg[\prod_{j=1}^{n}e^{W(e^{i\phi_{j}})}\prod_{k=0}^{m}|\mathsf{p}_{n}(t_{k})|^{\alpha_{k}}\prod_{k=1}^{m}e^{u_{k}N_{n}(\theta_{k})}\bigg] = \exp\left( \tilde{C}_{1} n + \tilde{C}_{2} \log n + \tilde{C}_{3} + \bigO \Big( n^{-1+\frac{u_{\max}}{\pi}} \Big)\right),
\end{align}
with $u_{\max} = \max \{ |\im u_{1}|,\dots,|\im u_{m}| \}$ and 
\begin{align}
\tilde{C}_{1} = &\; \sum_{k=0}^{m} \alpha_{k}\int_{0}^{2\pi} \log|e^{i\phi}-t_{k}|d\mu_V(e^{i\phi}) + \sum_{k=1}^{m} u_{k} \int_{0}^{\theta_{k}} d\mu_V(e^{i\phi}) 
+ \int_0^{2\pi} W(e^{i\phi})d\mu_V(e^{i\phi}), \\
\tilde{C}_{2} = &\; \sum_{k=0}^{m} \bigg( \frac{\alpha_{k}^{2}}{4}+\frac{u_{k}^{2}}{4\pi^{2}} \bigg), \\
\tilde{C}_{3} = &\;\sum_{\ell = 1}^{+\infty} \ell W_{\ell}W_{-\ell} - \sum_{k=0}^{m} \alpha_{k} \frac{W_{+}(t_{k})+W_{-}(t_{k})}{2} + \sum_{k=0}^{m} \frac{u_{k}}{\pi}\frac{W_{+}(t_{k}) -W_{-}(t_{k})}{2i} \\
& + \sum_{0 \leq j < k \leq m} \bigg\{ \frac{\alpha_{j} u_{k} - \alpha_{k} u_{j}}{4\pi}(\theta_{k}-\theta_{j }-\pi) - \bigg( \frac{u_{j}u_{k}}{2\pi^{2}}+\frac{\alpha_{j}\alpha_{k}}{2} \bigg) \log |t_{j}-t_{k}| \bigg\} \\
& + \sum_{k=0}^{m} \log \frac{G(1+\frac{\alpha_{k}}{2}+\frac{u_{k}}{2\pi i})G(1+\frac{\alpha_{k}}{2}-\frac{u_{k}}{2\pi i})}{G(1+\alpha_{k})} - \sum_{k=0}^m\frac{\frac{u_{k}^{2}}{\pi^{2}}+\alpha_{k}^{2}}{4\psi(t_k)}\left(\frac{1}{2\pi}-\psi(t_k)\right),
\end{align}
where $G$ is Barnes' $G$-function and $u_{0}:=-u_{1}-\dots-u_{m}$. Furthermore, the above asymptotics are uniform for all $\alpha_{k}$ in compact subsets of $\{z \in \mathbb{C}: \re z >-1\}$, for all $u_{k}$ in compact subsets of $\{z  \in \mathbb{C}: \im z \in (-\pi,\pi)\}$, and for all $(\theta_{1},\dots,\theta_{m})$ in compact subsets of $(0,2\pi)_{\mathrm{ord}}^{m}$. The above asymptotics can also be differentiated with respect to $\alpha_{0},\dots,\alpha_{m},u_{1},\dots,u_{m}$ as follows: if $k_{0},\dots,k_{2m}\in \mathbb{N}$, $k_{0}+\dots+k_{2m}\geq 1$ and $\partial^{\vec{k}}:=\partial_{\alpha_{0}}^{k_{0}}\dots\partial_{\alpha_{m}}^{k_{m}}\partial_{u_{1}}^{k_{m+1}}\dots\partial_{u_{m}}^{k_{2m}}$, then as $n \to + \infty$
\begin{align*}
\partial^{\vec{k}}\bigg( \log \mathbb{E}\bigg[\prod_{j=1}^{n}e^{W(e^{i\phi_{j}})}\prod_{k=0}^{m}|\mathsf{p}_{n}(t_{k})|^{\alpha_{k}}\prod_{k=1}^{m}e^{u_{k}N_{n}(\theta_{k})}\bigg] - \log \widehat{E}_{n} \bigg) = \bigO \bigg( \frac{(\log n)^{k_{m+1}+\dots+k_{2m}}}{n^{1-\frac{u_{\max}}{\pi}}} \bigg),
\end{align*}
where $\widehat{E}_{n}$ denotes the right-hand side of \eqref{exp in thm} without the error term.
\end{theorem}

Our first corollary is concerned with the smooth linear statistics of \eqref{point process}. For $V=0$, the central limit theorem stated in Corollary \ref{coro:smooth} was already obtained in \cite{Jo1988}.

\begin{corollary}\label{coro:smooth}(Smooth statistics.)
Let $V$ and $W$ be as in Theorem \ref{thm:MGF}, and assume furthermore that $W:\mathbb{T}\to \mathbb{R}$. Let $\{\kappa_{j}\}_{j=1}^{+\infty}$ be the cumulants of $\sum_{j=1}^{n}W(e^{i\phi_{j}})$, i.e.
\begin{align}\label{cum of W}
\kappa_{j} := \partial_{t}^{j} \log \mathbb{E}[e^{t \sum_{j=1}^{n}W(e^{i\phi_{j}})}]\big|_{t=0}.
\end{align}
As $n \to + \infty$, we have
\begin{align*}
& \mathbb{E}\bigg[\sum_{j=1}^{n}W(e^{i\phi_{j}})\bigg] = n \int_0^{2\pi} W(e^{i\phi})d\mu_V(e^{i\phi}) + \bigO \bigg( \frac{1}{n} \bigg), \\
& \mathrm{Var}\bigg[\sum_{j=1}^{n}W(e^{i\phi_{j}})\bigg] = 2\sum_{\ell = 1}^{+\infty} \ell W_{\ell}W_{-\ell} + \bigO \bigg( \frac{1}{n} \bigg), \\
& \kappa_{j} = \bigO \bigg( \frac{1}{n} \bigg), \qquad j \geq 3.
\end{align*}
Moreover, if $W$ is non-constant, then 
$$\frac{\sum_{j=1}^{n}W(e^{i\phi_{j}})-n\int_0^{2\pi} W(e^{i\phi})d\mu_V(e^{i\phi})}{(2\sum_{k = 1}^{+\infty} kW_{k}W_{-k})^{1/2}}$$
converges in distribution to a standard normal random variable.
\end{corollary}

Our second corollary considers
linear statistics for a test function with a $\log$-singularity at $t$. We let $\gamma_{\mathrm{E}}\approx 0.5772$ denote Euler's constant.

\begin{corollary}\label{coro:log}($\log |\cdot|$-statistics.)
Let $t=e^{i\theta} \in \mathbb{T}$ with $\theta\in [0,2\pi)$, and let $\{\kappa_{j}\}_{j=1}^{+\infty}$ be the cumulants of $\log |\mathsf{p}_{n}(t)|$, i.e.
\begin{align}\label{cum of log}
\kappa_{j} := \partial_{\alpha}^{j} \log \mathbb{E}[e^{\alpha \log |\mathsf{p}_{n}(t)|}]\big|_{\alpha=0}.
\end{align}
As $n \to + \infty$, we have
\begin{align*}
& \mathbb{E}[\log |\mathsf{p}_{n}(t)|] = n \int_{0}^{2\pi} \log|e^{i\phi}-t|d\mu_V(e^{i\phi}) + \bigO \bigg( \frac{1}{n} \bigg), \\
& \mathrm{Var}[\log |\mathsf{p}_{n}(t)|] = \frac{\log n}{2} + \frac{1+\gamma_{\mathrm{E}}}{2} - \frac{\frac{1}{2\pi}-\psi(t)}{2\psi(t)} + \bigO \bigg( \frac{1}{n} \bigg), \\
& \kappa_{j} = (-1+2^{1-j}) \; (\log G)^{(j)}(1) + \bigO \bigg( \frac{1}{n} \bigg), \qquad j \geq 3,
\end{align*}
and 
$$\frac{\log |\mathsf{p}_{n}(t)|-n\int_{0}^{2\pi} \log|e^{i\phi}-t|d\mu_V(e^{i\phi})}{\sqrt{\log n}/\sqrt{2}}$$ 
converges in distribution to a standard normal random variable.
\end{corollary}

Counting statistics of determinantal point processes have been widely studied over the years \cite{CostinLebowitz, SoshnikovSineAiryBessel} and is still a subject of active research, see e.g. the recent works \cite{SDMS2020, DXZ2022, CharlierAdv}. Our third corollary established various results on the counting statistics of \eqref{point process}.

\begin{corollary}\label{coro:counting}(Counting statistics.)
Let $t=e^{i\theta} \in \mathbb{T}$ be bounded away from $t_{0}:=1$, with $\theta\in (0,2\pi)$, and let $\{\kappa_{j}\}_{j=1}^{+\infty}$ be the cumulants of $N_{n}(\theta)$, i.e.
\begin{align}\label{cum of count}
\kappa_{j} := \partial_{u}^{j} \log \mathbb{E}[e^{u N_{n}(\theta)}]\big|_{u=0}.
\end{align}
As $n \to + \infty$, we have
\begin{align*}
& \mathbb{E}[N_{n}(\theta)] = n \int_{0}^{\theta} d\mu_V(e^{i\phi}) + \bigO \bigg( \frac{\log n}{n} \bigg), \\
& \mathrm{Var}[N_{n}(\theta)] = \frac{\log n}{\pi^{2}} + \frac{1+\gamma_{\mathrm{E}}+\log|t-1|}{\pi^{2}} - \frac{\frac{1}{2\pi}-\psi(1)}{2\pi^{2}\psi(1)} - \frac{\frac{1}{2\pi}-\psi(t)}{2\pi^{2}\psi(t)} + \bigO \bigg( \frac{(\log n)^{2}}{n} \bigg), \\
& \kappa_{2j+1} = \bigO \bigg( \frac{(\log n)^{2j+1}}{n} \bigg), \qquad j \geq 1, \\
& \kappa_{2j+2} = \frac{(-1)^{j+1}}{2^{2j}\pi^{2j+2}} \; (\log G)^{(2j+2)}(1) + \bigO \bigg( \frac{(\log n)^{2j+2}}{n} \bigg), \qquad j \geq 1,
\end{align*}
and $\frac{N_{n}(\theta)-n\int_{0}^{\theta} d\mu_V(e^{i\phi})}{\sqrt{\log n}/\pi}$ converges in distribution to a standard normal random variable.
\end{corollary}

\begin{remark}
There are several differences between smooth, $\log$- and counting statistics that are worth pointing out:
\begin{itemize}
\item \vspace{-0.15cm} The variance of the smooth statistics is of order $1$, while the variances of the $\log$- and counting statistics are of order $\log n$.
\item \vspace{-0.15cm} The third and higher order cumulants of the smooth statistics are all $\bigO(n^{-1})$, while for the $\log$-statistics the corresponding cumulants are all of order $1$. On the other hand, the third and higher order cumulants of the counting statistics are as follows: the odd cumulants are $o(1)$, while the even cumulants are of order $1$. This phenomena for the counting statistics was already noticed in \cite[eq (29)]{SDMS2020} for a class of determinantal point processes.
\end{itemize}
\end{remark}

Another consequence of Theorem \ref{thm:MGF} is the following result about the individual fluctuations of the ordered angles. Corollary \ref{coro:order} is an analogue for \eqref{point process} of Gustavsson's well-known result \cite[Theorem 1.2]{Gustavsson} for the Gaussian unitary ensemble.

\begin{corollary}\label{coro:order}(Ordered statistics.)
Let $\xi_{1}\leq \xi_{2} \leq \dots \leq \xi_{n}$ denote the ordered angles, \vspace{0.1cm}
\begin{align}\label{xijdef}
\xi_{1}=\min\{\phi_{1},\dots,\phi_{n}\}, \quad \xi_{j} = \inf_{\theta\in [0,2\pi)}\{\theta:N_{n}(\theta)=j\}, \quad j=1,\dots,n,
\end{align} 
and let $\eta_{k}$ be the classical location of the $k$-th smallest angle $\xi_{k}$, 
\begin{align}\label{def of kappa k}
\int_{0}^{\eta_{k}}d\mu_V(e^{i\phi}) = \frac{k}{n}, \qquad k=1,\dots,n.
\end{align}
Let $t=e^{i\theta} \in \mathbb{T}$ with $\theta\in (0,2\pi)$. Let $k_{\theta}=[n \int_{0}^{\theta}d\mu_V(e^{i\phi})]$, where $[x]:= \lfloor x + \frac{1}{2}\rfloor$ is the closest integer to $x$. As $n \to + \infty$, $\frac{n\psi(e^{i\eta_{k_{\theta}}})}{\sqrt{\log n}/\pi}(\xi_{k_{\theta}}-\eta_{k_{\theta}})$ converges in distribution to a standard normal random variable.
\end{corollary}

There has been a lot of progress in recent years towards understanding the global rigidity of various point processes, see e.g. \cite{ErdosYauYin, ABB2017, CFLW2021}. Our next corollary is a contribution in this direction: it establishes global rigidity upper bounds for (i) the counting statistics of \eqref{point process} and (ii) the ordered statistics of \eqref{point process}.

\begin{corollary}\label{coro:rigidity}(Rigidity.)
For each $\epsilon >0$ sufficiently small, there exist $c>0$ and $n_{0}>0$ such that
\begin{align}
& \mathbb P\left(\sup_{0 \leq \theta < 2\pi}\bigg|N_{n}(\theta)- n\int_{0}^{\theta}d\mu_V(e^{i\phi})  \bigg|\leq (1+\epsilon)\frac{1}{\pi}\log n \right) \geq 1-\frac{c}{\log n}, \label{probabilistic upper bound 1} \\
& \mathbb{P}\bigg( \max_{1 \leq k \leq n}  \psi(e^{i\eta_{k}})|\xi_{k}-\eta_{k}| \leq (1+\epsilon)\frac{1}{\pi} \frac{\log n}{n} \bigg) \geq 1-\frac{c}{\log n}, \label{probabilistic upper bound 2}
\end{align}
for all $n \geq n_{0}$.
\end{corollary}

\begin{remark}
It follows from \eqref{probabilistic upper bound 2} that $\lim_{n\to \infty}\mathbb{P}\big( \max_{1 \leq k \leq n}  \psi(e^{i\eta_{k}})|\xi_{k}-\eta_{k}| \leq (1+\epsilon)\frac{1}{\pi}\frac{\log n}{n} \big) = 1$. We believe that the upper bound $(1+\epsilon)\frac{1}{\pi}$ is sharp, in the sense that we expect the following to hold true:
\begin{align}\label{lol5}
\lim_{n\to +\infty}\mathbb{P}\bigg( (1-\epsilon)\frac{1}{\pi}\frac{\log n}{n}  \leq \max_{1 \leq k \leq n}  \psi(e^{i\eta_{k}})|\xi_{k}-\eta_{k}| \leq (1+\epsilon)\frac{1}{\pi} \frac{\log n}{n}  \bigg) = 1.
\end{align}
Our belief is supported by the fact that \eqref{lol5} was proved in \cite[Theorem 1.5]{ABB2017} for $V=0$, $\psi(e^{i\theta})=\frac{1}{2\pi}$.
\end{remark}

\section{Differential identity for $D_n$}
Our general strategy to prove Theorem \ref{theorem U} is inspired by the earlier works \cite{Krasovsky, DIK2011, BerWebbWong, Charlier}. The first step consists of establishing a differential identity which expresses derivatives of $\log D_n(\vec\alpha,\vec\beta,V,W)$ in terms of the solution $Y$ to a Riemann-Hilbert (RH) problem (see Proposition \ref{prop: diff id}). Throughout the paper, $\mathbb{T}$ is oriented in the counterclockwise direction. We first state the RH problem for $Y$.

\subsubsection*{RH problem for $Y(\cdot) = Y_n(\cdot;\vec\alpha,\vec\beta,V,W)$}
\begin{itemize}
\item[(a)] $Y : \C \setminus \mathbb{T} \to \mathbb{C}^{2 \times 2}$ is analytic.
	
\item[(b)] For each $z \in \mathbb{T}\setminus \{t_{0},\dots,t_{m}\}$, the boundary values $\lim_{z' \to z}Y(z')$ from the interior and exterior of $\mathbb{T}$ exist, and are denoted by $Y_{+}(z)$ and $Y_{-}(z)$ respectively. Furthermore, $Y_{+}$ and $Y_{-}$ are continuous on $\mathbb{T}\setminus \{t_{0},\dots,t_{m}\}$, and are related by the jump condition 
\begin{equation}\label{Y jump}
Y_+(z) = Y_-(z)\begin{pmatrix}1& z^{-n}f(z) \\ 0& 1\end{pmatrix}, \qquad z \in \mathbb{T}\setminus \{t_{0},\dots,t_{m}\},
\end{equation}
where $f$ is given by \eqref{uni: f}.

\item[(c)] $Y$ has the following asymptotic behavior at infinity: 
\begin{align*}
Y(z) = (1+\bigO(z^{-1}))z^{n \sigma_{3}}, \qquad \mbox{as } z \to \infty,
\end{align*}
where $\sigma_{3} = \begin{pmatrix} 1 & 0 \\ 0 & -1 \end{pmatrix}$. 
\item[(d)] As $z \rightarrow t_k$, $k=0, \dots, m$, $z \in \C \setminus \mathbb{T}$, 
\begin{align*}
Y(z) = \begin{cases}
\begin{pmatrix}
\bigO(1) & \bigO(1) + \bigO({|z-t_k|}^{\alpha_{k}}) \\ \bigO(1) & \bigO(1) + \bigO({|z-t_k|}^{\alpha_{k}})
\end{pmatrix}, &\mbox{if } \re \alpha_k \ne 0, \\
\begin{pmatrix}
\bigO(1) & \bigO(\log{|z-t_k|}) \\ \bigO(1) & \bigO(\log{|z-t_k|})
\end{pmatrix}, &\mbox{if } \re \alpha_k = 0.
\end{cases}
\end{align*}
\end{itemize}
Suppose $\{p_k(z) = \kappa_{k} z^{k}+\dots\}_{k\geq 0}$ and $\{\hat{p}_k(z)=\kappa_{k} z^{k}+\dots\}_{k\geq 0}$ are two families of polynomials satisfying the orthogonality conditions
\begin{equation}\label{orthogonality relations}
\begin{cases}
\frac{1}{2\pi}\int_0^{2\pi} p_k(z)z^{-j}f(z)d\theta= \kappa_k^{-1}\delta_{jk}, \\
\frac{1}{2\pi}\int_0^{2\pi} \hat{p}_k(z^{-1})z^{j}f(z)d\theta= \kappa_k^{-1}\delta_{jk},
\end{cases}
\quad z=e^{i\theta}, \quad j= 0,\dots,k.
\end{equation}
Then the function $Y(z)$ defined by
\begin{equation}\label{Y_solution}
Y(z) = \begin{pmatrix}
\kappa_n^{-1}p_n(z) & \kappa_n^{-1}\int_{\mathbb{T}} \frac{p_n(s)f(s)}{2\pi i s^n(s-z)}ds \\
-\kappa_{n-1}z^{n-1}\hat{p}_{n-1}(z^{-1}) & 
-\kappa_{n-1}\int_{\mathbb{T}} \frac{\hat{p}_{n-1}(s^{-1})f(s)}{2\pi i s(s-z)}ds
\end{pmatrix}
\end{equation}
solves the RH problem for $Y$.
It was first noticed by Fokas, Its and Kitaev \cite{FokasItsKitaev} that orthogonal polynomials can be characterized by RH problems (for a contour on the real line). The above RH problem for $Y$, whose jumps lie on the unit circle, was already considered in e.g. \cite[eq. (1.26)]{BaikDeiftJohansson} and \cite[eq. (3.1)]{DIK2011} for more specific $f$. 

The monic orthogonal polynomials $\kappa_{n}^{-1}p_{n}, \kappa_{n}^{-1}\hat{p}_{n}$, and also $Y$, are unique (if they exist). The orthogonal polynomials exist if $f$ is strictly positive almost everywhere on $\mathbb{T}$ (this is the case if $W$ is real-valued, $\alpha_{k}>-1$ and $i\beta_{k} \in (-\frac{1}{2},\frac{1}{2})$). More generally, a sufficient condition to ensure existence of $p_{n}, \hat{p}_{n}$ (and therefore of $Y$) is that $D_{n}^{(n)} \neq 0 \neq D_{n+1}^{(n)}$, where $D_{l}^{(n)} =: \det (f_{j-k})_{j,k=0,\dots,l-1}$, $l\geq 1$ (note that $D_{n}^{(n)}=D_{n}(\vec{\alpha},\vec{\beta},V,W)$), see e.g. \cite[Section 2.1]{ClKr}. In fact, 
\begin{equation}\label{uni: pn}
p_{k}(z) = \frac{\begin{vmatrix}
f_{0} & f_{-1} & \dots & f_{-k}\\
\vdots & \vdots & \ddots & \vdots\\
f_{k-1}& f_{k-2} & \dots & f_{-1}\\
1 & z & \dots & z^k
\end{vmatrix}}{\sqrt{D_k^{(n)}}\sqrt{D_{k+1}^{(n)}}}, \quad \hat{p}_k(z) = \frac{\begin{vmatrix}
f_{0} & f_{-1} & \dots & f_{-k+1} & 1 \\
f_{1} & f_{0} & \dots & f_{-k+2} & z \\
\vdots & \vdots &  & \vdots & \vdots \\
f_{k}& f_{k-1} & \dots & f_{1} & z^{k} \\
\end{vmatrix}}{\sqrt{D_k^{(n)}}\sqrt{D_{k+1}^{(n)}}},
\end{equation}
and $\kappa_k= (D_{k}^{(n)})^{1/2}/(D_{k+1}^{(n)})^{1/2}$. (Note that $p_{k}$, $\hat{p}_{k}$ and $\kappa_{k}$ are unique only up to multiplicative factors of $-1$. This can be fixed with a choice of the branch for the above roots. However, since $Y$ only involves $\kappa_{n}^{-1}p_{n}$ and $\kappa_{n-1}\hat{p}_{n-1}$, which are unique, this choice for the branch is unimportant for us.) If $D_{k}^{(n)}\neq 0$ for $k=0,1,\dots,n+1$, it follows that
\begin{equation}\label{prod_det}
D_n(\vec{\alpha},\vec{\beta},V,W) = \prod_{j=0}^{n-1} \kappa_j^{-2}. 
\end{equation}

\begin{lemma}\label{lemma: YinvYprime21}
Let $n \in \mathbb{N}$ be fixed, and assume that $D_k^{(n)}(f)\neq 0$, $k = 0,1,\dots,n+1$. For any $z\ne 0$, we have
\begin{equation}
[Y^{-1}(z)Y'(z)]_{21}z^{-n+1} = \sum_{k = 0}^{n-1}\hat{p}_k(z^{-1})p_k(z),
\end{equation}
where $Y(\cdot) = Y_n(\cdot;\vec\alpha,\vec\beta,V,W)$.
\end{lemma}
\begin{proof}
The assumptions imply that $\kappa_k= (D_k^{(n)})^{1/2}/(D_{k+1}^{(n)})^{1/2}$ is finite and nonzero and that $p_k, \hat{p}_k$ exist for all $k \in \{0,\dots,n\}$. Note that (a) $\det Y: \C \setminus \T \to \C$ is analytic, (b) $(\det Y)_{+}(z) = (\det Y)_{-}(z)$ for $z \in \T\setminus\{t_{0},\ldots,t_{m}\}$, (c) $\det Y(z) = o(|z-t_{k}|^{-1})$ as $z \to t_{k}$ and (d) $\det Y(z) = 1+o(1)$ as $z\to \infty$. Hence, using successively Morera's Theorem, Riemann's removable singularities theorem, and Liouville's theorem, we conclude that $\det Y \equiv 1$. Using \eqref{Y_solution} and the fact that $\det Y \equiv 1$, we obtain
\begin{align*}
\left[Y^{-1}(z)Y'(z)\right]_{21}&=\frac{z^n}{\kappa_n}\cdot\frac{\kappa_{n-1}}{z}\hat{p}_{n-1}(z^{-1})\frac{d}{dz}p_n(z)-\kappa_n^{-1}p_n(z)\frac{d}{dz}\left[z^n\cdot\frac{\kappa_{n-1}}{z}\hat{p}_{n-1}(z^{-1})\right].
\end{align*}
Using the recurrence relation (see \cite[Lemma 2.2]{DIK2011})
\begin{align*}
\frac{\kappa_{n-1}}{z}\hat{p}_{n-1}(z^{-1})&= \kappa_{n}\hat{p}_{n}(z^{-1})-\hat{p}_{n}(0)z^{-n}p_{n}(z),
\end{align*}
we then find
\begin{align*}
\left[Y^{-1}(z)Y'(z)\right]_{21} = z^{n-1}\left(-np_n(z)\hat{p}_n(z^{-1})+z\Big(\hat{p}_n(z^{-1})\frac{d}{dz}p_n(z)-p_n(z)\frac{d}{dz}\hat{p}_n(z^{-1})\Big)\right).
\end{align*}
The claim now directly follows from the Christoffel-Darboux formula \cite[Lemma 2.3]{DIK2011}.
\end{proof}

\begin{proposition}\label{prop: diff id}
Let $n \in \mathbb{N}_{\geq 1}:=\{1,2,\dots\}$ be fixed and suppose that $f$ depends smoothly on a parameter $\gamma$. If $D_k^{(n)}(f)\neq 0$ for $k = n-1,n,n+1$, then the following differential identity holds
\begin{align}\label{uni: diff id}
\partial_{\gamma} \log D_n(\vec{\alpha},\vec{\beta},V,W) = \frac{1}{2\pi}\int_{0}^{2\pi}[Y^{-1}(z)Y'(z)]_{21}z^{-n+1}\partial_{\gamma}f(z)d\theta, \qquad z=e^{i\theta}.
\end{align}
\end{proposition}
\begin{remark}
Identity \eqref{uni: diff id} will be used (with a particular choice of $\gamma$) in the proof of Proposition \ref{prop: int V}  to deform the potential, see \eqref{int: V}.
\end{remark}
\begin{proof}
We first prove the claim under the stronger assumption that $D_k^{(n)}(f)\neq 0$ for $k = 0,1,\dots,n+1$. 
In this case, $\kappa_k= (D_k^{(n)})^{1/2}/(D_{k+1}^{(n)})^{1/2}$ is finite and nonzero and $p_k, \hat{p}_k$ exist for all $k = 0,1,\dots,n$.
Replacing $z^{-j}$ with $\hat{p}_{j}(z^{-1})\kappa_j^{-1}$ in the first orthogonality condition in \eqref{orthogonality relations} (with $k=j$), and differentiating with respect to $\gamma$, we obtain, for $j = 0, \dots, n-1$,
\begin{align}\nonumber
    -\frac{\partial_\gamma[\kappa_j]}{\kappa_j}&=\frac{\kappa_j}{2\pi}\partial_\gamma\left[\int_0^{2\pi}p_j(z)\hat{p}_j(z^{-1})\kappa_j^{-1}f(z)d\theta\right] 
    	\\\label{partialgammakappaj}
    &=\frac{1}{2\pi}\int_0^{2\pi}p_j(z)\hat{p}_j(z^{-1})\partial_\gamma[f(z)]d\theta+\frac{\kappa_j}{2\pi}\int_0^{2\pi}\partial_\gamma\left[p_j(z)\hat{p}_j(z^{-1})\kappa_j^{-1}\right]f(z)d\theta.
\end{align}
The second term on the right-hand side can be simplified as follows:  
\begin{align}\label{kappajintint}
\frac{\kappa_j}{2\pi}\int_0^{2\pi}\partial_\gamma\left[p_j(z)\hat{p}_j(z^{-1})\kappa_j^{-1}\right]f(z)d\theta = \frac{\kappa_j}{2\pi}\int_0^{2\pi}\partial_\gamma[p_j(z)]\hat{p}_j(z^{-1})\kappa_j^{-1}f(z)d\theta = \frac{\partial_\gamma[\kappa_j]}{\kappa_j},
\end{align}
where the first and second equalities use the first and second relations in \eqref{orthogonality relations}, respectively. Combining \eqref{partialgammakappaj} and \eqref{kappajintint}, we find
\begin{align}
    -2\frac{\partial_\gamma[\kappa_j]}{\kappa_j}&=\frac{1}{2\pi}\int_0^{2\pi}p_j(z)\hat{p}_j(z^{-1})\partial_\gamma[f(z)]d\theta.
\end{align}
Taking the log of both sides of \eqref{prod_det} and differentiating with respect to $\gamma$, we get
\begin{align}
    \partial_{\gamma}\log D_{n}(\vec{\alpha},\vec{\beta},V,W)&=-2\sum_{j=0}^{n-1}\frac{\partial_\gamma[\kappa_j]}{\kappa_j}=\frac{1}{2\pi}\int_0^{2\pi}\bigg(\sum_{j=0}^{n-1}p_j(z)\hat{p}_j(z^{-1})\bigg)\partial_\gamma[f(z)]d\theta.
\end{align}
An application of Lemma \ref{lemma: YinvYprime21} completes the proof under the assumption that $D_k^{(n)}(f)\neq 0$, $k = 0,1,\dots,n+1$. Since the existence of $Y$ only relies on the weaker assumption $D_k^{(n)}(f)\neq 0$, $k = n-1,n,n+1$, 
the claim follows from a simple continuity argument.
\end{proof}

\section{Steepest descent analysis}\label{section:steepest descent analysis}

In this section, we use the Deift-Zhou \cite{DZ1993} steepest descent method to obtain large $n$ asymptotics for $Y$.

\subsection{Equilibrium measure and $g$-function}
The first step of the method is to normalize the RH problem at $\infty$ by means of a so-called $g$-function built in terms of the equilibrium measure \eqref{uni: dmu}. Recall from \eqref{V as a Laurent series}, \eqref{W as a Laurent series} and \eqref{psi solved explicitly} that $U$ is an open annulus containing $\mathbb{T}$ in which $V$, $W$ and $\psi$ are analytic.

Define the function $g:\mathbb{C}\setminus \big((-\infty,-1]\cup \mathbb{T} \big)\to\mathbb{C}$ by
\begin{equation}\label{g-def}
g(z) = \int_{\mathbb{T}} \log_{s} (z-s) \psi(s) \frac{ds}{is},
\end{equation}
where for $s = e^{i \theta} \in \mathbb{T}$ and $\theta \in [-\pi,\pi)$, the function $z \mapsto \log_{s} (z-s)$ is analytic in $\mathbb{C}\setminus \big((-\infty,-1]\cup \{e^{i \theta'}: -\pi \leq \theta' \leq \theta \}\big)$ and such that $\log_{s} (2)=\log |2|$. 

\begin{lemma}
The function $g$ defined in (\ref{g-def}) is analytic in $\mathbb{C}\setminus \big((-\infty,-1]\cup \mathbb{T} \big)$, satisfies $g(z) = \log z + \bigO(z^{-1})$ as $z \to \infty$, and possesses the following properties:
\begin{align}
& g_{+}(z) + g_{-}(z) = 2 \int_{\mathbb{T}} \log |z-s| \psi(s)\frac{ds}{is} + i\big(\pi + \hat{c} + \arg z \big), & & z \in \mathbb{T}, \label{g+ + g-} \\
& g_{+}(z) - g_{-}(z) = 2 \pi i \int_{\arg z}^{\pi} \psi(e^{i\theta})d\theta, & & z \in \mathbb{T}, \label{g+ - g- 2} 
	\\ 
& g_{+}(z) - g_{-}(z) = 2 \pi i, & & z \in (-\infty,-1), \label{g+ - g- 1}
\end{align}
where $\hat{c}=\int_{-\pi}^{\pi} \theta \psi(e^{i\theta}) d\theta$ and $\arg z \in (-\pi,\pi)$. 
\end{lemma}
\begin{proof}
In the case where the equilibrium measure satisfies the symmetry $\psi(e^{i\theta})=\psi(e^{-i\theta})$, we have $\hat{c}=0$ and in this case \eqref{g+ + g-}--\eqref{g+ - g- 1} follow from \cite[Lemma 4.2]{BaikDeiftJohansson}. In the more general setting of a non-symmetric equilibrium measure, \eqref{g+ + g-}--\eqref{g+ - g- 1} can be proved along the same lines as \cite[proof of Lemma 4.2]{BaikDeiftJohansson} (the main difference is that $F(\pi)=\pi$ in \cite[proof of Lemma 4.2]{BaikDeiftJohansson} should here be replaced by $F(\pi)=\pi+\hat{c}$).
\end{proof}

It follows from \eqref{g+ - g- 2} that
\begin{align}\label{uni: gprime}
    g'_+(z)-g'_-(z)=-\frac{2\pi}{z}\psi(z), ~~~ z\in \mathbb{T}.
\end{align}
Substituting \eqref{g+ + g-} into the Euler-Lagrange equality \eqref{var equality} and recalling that $d\mu_V(s) = \psi(s) \frac{ds}{is}$, we get
\begin{align}\label{EL= in terms of g}
	V(z) = g_{+}(z) + g_{-}(z) + \ell - \log z - i\big(\pi+ \hat{c} \big), \qquad z \in \mathbb{T},
\end{align}
where the principal branch is taken for the logarithm. Consider the function 
\begin{equation}\label{integral form of xi}
	\xi(z) = \begin{cases}
	\ds -i \pi \int_{-1}^{z} \psi(s) \frac{ds}{is}, & \mbox{if } |z|<1, \; z \in U, \\
	\ds i \pi \int_{-1}^{z} \psi(s) \frac{ds}{is}, & \mbox{if } |z|>1, \; z \in U,
	\end{cases}
\end{equation}
where the contour of integration (except for the starting point $-1$) lies in $U \setminus \big((-\infty,0]\cup \mathbb{T}\big)$ and the first part of the contour lies in $\{z : \im z \geq 0\}$. Since $\psi$ is real-valued on $\mathbb{T}$, we have $\re \xi(z)=0$ for $z \in \mathbb{T}$. Using the Cauchy-Riemann equations in polar coordinates and the compactness of the unit circle, we verify that there exists an open annulus $U' \subseteq U$ containing $\mathbb{T}$ such that $\re \xi(z) > 0$ for $z \in U'\setminus \mathbb{T}$. Redefining $U$ if necessary, we can (and do) assume that $U'=U$. Furthermore, for $z = e^{i\theta} \in \mathbb{T}$, $\theta \in (-\pi,\pi)$, we have
\begin{align}
	& \xi_{+}(z)-\xi_{-}(z) = 2\xi_{+}(z) = -2  \pi i \int_{-1}^{z} \psi(s) \frac{ds}{is} = 2  \pi i \int_{\theta}^{\pi} \psi(e^{i\theta'}) d\theta' = g_{+}(z)-g_{-}(z), \label{prop of xi} \\
	& 2\xi_{\pm}(z) - 2g_{\pm}(z) =  \ell - V(z)-\log z-i\pi -i \hat{c}. \label{xi +}
\end{align}
Analytically continuing $\xi(z)-g(z)$ in \eqref{xi +}, we obtain
\begin{equation}\label{relation between g and xi}
	\xi(z) = g(z) + \frac{1}{2}\left(\ell - V(z)-\log z-i\pi -i\hat{c}\right), \quad \mbox{for all } z \in U\setminus \big( (-\infty,0]\cup \mathbb{T} \big).
\end{equation}
Note also that 
\begin{align}
& \xi_{+}(x) - \xi_{-}(x) = \pi i, & & x \in U\cap (-\infty,-1), \label{jump xi neg 1} \\
& \xi_{+}(x) - \xi_{-}(x) = -\pi i, & & x \in U\cap (-1,0), \label{jump xi neg 2}
\end{align}
where $\xi_{\pm}(x) := \lim_{\epsilon \to 0^{+}}\xi(z\pm i\epsilon)$ for $x \in U\cap((-\infty,-1)\cup (-1,0))$.

\subsection{Transformations $Y\rightarrow T\rightarrow S$}
The first transformation $Y\rightarrow T$ is defined by
\begin{equation}\label{def of T}
T(z) = e^{-\frac{n(\pi+\hat{c}) i}{2}\sigma_{3}}e^{\frac{n\ell}{2}\sigma_3}Y(z)e^{-ng(z)\sigma_3}e^{-\frac{n\ell}{2}\sigma_3}e^{\frac{n(\pi+\hat{c}) i}{2}\sigma_{3}}.
\end{equation}
For $z \in \mathbb{T} \setminus \{t_{0},\dots,t_{m}\}$, the function $T$ satisfies the jump relation $T_+=T_-J_T$ where the jump matrix $J_T$ is given by
\begin{align*}
	J_T(z) = \begin{pmatrix}
    e^{-n(g_{+}(z)-g_{-}(z))} & z^{-n}e^{-n[V(z)-g_{+}(z)-g_{-}(z)-\ell+i\pi+i\hat{c}]}e^{W(z)}\omega(z) \\
	0 & e^{n(g_{+}(z)-g_{-}(z))}
	\end{pmatrix}.
\end{align*}
Combining the above with \eqref{g+ - g- 1}, \eqref{EL= in terms of g} and \eqref{prop of xi}, we conclude that $T$ satisfies the following RH problem. 
\subsubsection*{RH problem for $T$}
\begin{itemize}
\item[(a)] $T: \C \setminus \mathbb{T} \rightarrow \C^{2\times2}$ is analytic.
\item[(b)] The boundary values $T_{+}$ and $T_{-}$ are continuous on $\mathbb{T}\setminus \{t_{0},\dots,t_{m}\}$ and are related by
\begin{align*}
T_+(z) = T_-(z)\begin{pmatrix}
e^{-2n\xi_{+}(z)} & e^{W(z)} \omega(z) \\ 0 & e^{-2n\xi_{-}(z)} 
\end{pmatrix}, \qquad z \in \mathbb{T} \setminus \{t_{0},\dots,t_{m}\}.
\end{align*}
\item[(c)] As $z \to \infty$, $T(z) = I + \bigO(z^{-1})$.
\item[(d)] As $z \rightarrow t_k$, $k = 0, \dots, m$, $z \in \C \setminus \mathbb{T}$, 
\begin{align*}
T(z) = \begin{cases}
\begin{pmatrix}
\bigO(1) & \bigO(1) + \bigO({|z-t_k|}^{\alpha_{k}}) \\ \bigO(1) & \bigO(1) + \bigO({|z-t_k|}^{\alpha_{k}})
\end{pmatrix}, &\mbox{if } \re \alpha_k \ne 0, \\
\begin{pmatrix}
\bigO(1) & \bigO(\log{|z-t_k|}) \\ \bigO(1) & \bigO(\log{|z-t_k|})
\end{pmatrix}, &\mbox{if } \re \alpha_k = 0.
\end{cases}
\end{align*}
\end{itemize}
The jumps of $T$ for $z \in \mathbb{T}\setminus\{t_{0},\dots,t_{m}\}$ can be factorized as
\begin{multline*}
\begin{pmatrix}
e^{-2n\xi_{+}(z)} & e^{W(z)} \omega(z) \\
0 & e^{-2n\xi_{-}(z)}
\end{pmatrix} = \begin{pmatrix}
1 & 0 \\ e^{-W(z)}\omega(z)^{-1}e^{-2n \xi_{-}(z)} & 1
\end{pmatrix} \\ \times \begin{pmatrix}
0 & e^{W(z)}\omega(z) \\ -e^{-W(z)}\omega(z)^{-1} & 0
\end{pmatrix} \begin{pmatrix}
1 & 0 \\ e^{-W(z)}\omega(z)^{-1}e^{-2n \xi_{+}(z)} & 1
\end{pmatrix}.
\end{multline*}
Before proceeding to the second transformation, we first describe the analytic continuations of the functions appearing in the above factorization. The functions $\omega_{\beta_{k}}$, $k =0,\dots,m$, have a straightforward analytic continuation from $\mathbb{T} \setminus \{t_{k}\}$ to $\mathbb{C}\setminus \{\lambda t_{k}: \lambda \geq 0\}$, which is given by
\begin{align}\label{omega beta ana cont}
& \omega_{\beta_{k}}(z) = z^{\beta_{k}}t_{k}^{-\beta_{k}}\times \begin{cases}
e^{i \pi \beta_{k}}, & 0 \leq \arg_{0} z < \theta_{k}, \\
e^{-i \pi \beta_{k}}, & \theta_{k} \leq \arg_{0} z < 2 \pi,
\end{cases}  & & z \in \mathbb{C}\setminus \{\lambda t_{k}: \lambda \geq 0\}, \; k =0,\dots,m,
\end{align}
where $\arg_{0} z \in [0,2\pi)$, $t_{k}^{-\beta_{k}} := e^{-i\beta_{k} \theta_{k}}$, and $z^{\beta_{k}} := |z|^{\beta_{k}}e^{i\beta_{k} \arg_{0} z}$. For the root-type singularities, we follow \cite{DIK2011} and analytically continue $\omega_{\alpha_{k}}$ from $\mathbb{T}\setminus \{t_{k}\}$ to $\mathbb{C}\setminus \{\lambda t_{k}: \lambda \geq 0\}$ as follows
\begin{align*}
	& \omega_{\alpha_{k}}(z) = \frac{(z-t_{k})^{\alpha_{k}}}{(z t_{k} e^{i \ell_{k}(z)})^{\alpha_{k}/2}} := \frac{e^{\alpha_{k}(\log |z-t_{k}|+i \hat{\arg}_{k}(z-t_{k}))}}{e^{\frac{\alpha_{k}}{2}(\log |z| + i \arg_{0}(z) + i \theta_{k} + i \ell_{k}(z))}}, \quad z \in \mathbb{C}\setminus \{\lambda t_{k}: \lambda \geq 0\}, \; k=0,\dots,m,
\end{align*}
where $\hat{\arg}_{k} z \in (\theta_{k},\theta_{k}+2\pi)$, and
\begin{align*}
\ell_{k}(z) = \begin{cases}
3 \pi, & 0 \leq \arg_{0} z < \theta_{k}, \\
\pi, & \theta_{k} \leq \arg_{0} z < 2 \pi.
\end{cases}
\end{align*}
Now, we open lenses around $\mathbb{T}\setminus \{t_{0},\dots,t_{m}\}$ as shown in Figure \ref{fig: opening of the lenses}. The part of the lens-shaped contour lying in $\{|z|<1\}$ is denoted $\gamma_{+}$, and the part lying in $\{|z|>1\}$ is denoted $\gamma_{-}$. We require that $\gamma_{+},\gamma_{-} \subset U$.

	\begin{figure}
\vspace{-0.5cm}
		\begin{center}
			\begin{tikzpicture}
			\node at (0,0) {};
			
			\draw[fill] (0:2.5) circle (0.075cm); 
			\draw[fill] (80:2.5) circle (0.075cm); 
			\draw[fill] (230:2.5) circle (0.075cm); 
			
			\draw[black,line width=0.5 mm,->-=0.33,->-=0.62,->-=0.94] ([shift=(-180:2.5cm)]0,0) arc (-180:180:2.5cm);

			\draw[black, line width=0.5 mm,->-=0.12,->-=0.445,->-=0.83] (0:2.5) to [out=45, in=40-90] (40:3) to [out=40+90, in=35] (80:2.5) to [out=125, in=-90+155] (155:3.1) to [out=155+90, in=90+40+45] (230:2.5) to [out=-50-35, in=300-90] (300:3) to [out=300+90, in=-45] (0:2.5);

			\draw[black, line width=0.5 mm,->-=0.12,->-=0.445,->-=0.83] (0:2.5) to [out=45+90, in=35-90] (80:2.5) to [out=170+45, in=-90+152] (155:1.9) to [out=152+90, in=90+50+45-90] (230:2.5) to [out=-40-45+90, in=290-90] (290:2) to [out=290+90, in=-45-90] (00:2.5);
			
%
			
			\node at (0:2.85) {$t_{0}$};
			\node at (80:2.8) {$t_{1}$};
			\node at (230:2.8) {$t_{2}$};

			\node at (33:1.55) {$\gamma_+$};
			\node at (152:1.52) {$\gamma_+$};
			\node at (295:1.7) {$\gamma_+$};

			\node at (36:3.34) {$\gamma_-$};
			\node at (155:3.42) {$\gamma_-$};
			\node at (298:3.34) {$\gamma_-$};
			\end{tikzpicture}
		\end{center}
		\caption{\label{fig: opening of the lenses}
			The jump contour for $S$ with $m=2$.}
	\end{figure}
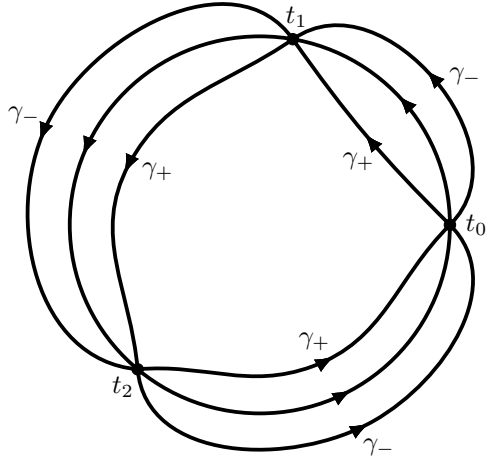
The transformation $T \mapsto S$ is defined by
\begin{align}\label{T to S}
	S(z) = T(z) \times \begin{cases}
	I, & \mbox{if } z \mbox{ is outside the lenses,} \\
	\begin{pmatrix}
	1 & 0 \\ e^{-W(z)}\omega(z)^{-1}e^{-2n\xi(z)}& 1 
	\end{pmatrix}, & \mbox{if } |z| > 1 \mbox{ and inside the lenses,}
	\\
	\begin{pmatrix}
	1 & 0 \\ -e^{-W(z)}\omega(z)^{-1}e^{-2n\xi(z)}& 1 
	\end{pmatrix}, & \mbox{if } |z| < 1 \mbox{ and inside the lenses.}
	\end{cases}
\end{align}
Note from \eqref{jump xi neg 1}--\eqref{jump xi neg 2} that $e^{-2n\xi(z)}$ is analytic in $U \cap ((-\infty,-1)\cup(-1,0))$. It can be verified using the RH problem for $T$ and \eqref{T to S} that $S$ satisfies the following RH problem.
\subsubsection*{RH problem for $S$}
\begin{itemize}
\item[(a)] $S : \C \setminus (\gamma_{+} \cup \gamma_{-} \cup \mathbb{T}) \to \mathbb{C}^{2 \times 2}$ is analytic, where $\gamma_+, \gamma_-$ are the contours in Figure \ref{fig: opening of the lenses} lying inside and outside $\mathbb{T}$, respectively.
\item[(b)] The jumps for $S$ are as follows. 
\begin{align*}
& S_+(z) = S_-(z)\begin{pmatrix}
0 & e^{W(z)}\omega(z) \\ -e^{-W(z)}\omega(z)^{-1} & 0
\end{pmatrix}, && z \in \mathbb{T}\setminus \{t_{0},\dots,t_{m}\},
\\
& S_+(z) = S_-(z)\begin{pmatrix}
1 & 0 \\ e^{-W(z)}\omega(z)^{-1}e^{-2n\xi(z)} & 1
\end{pmatrix}, && z \in \gamma_{+}\cup \gamma_{-}.
\end{align*}
\item[(c)] As $z \rightarrow \infty$, $S(z) = I + \bigO(z^{-1})$.
\item[(d)] As $z \to t_{k}$, $k = 0,\dots,m$, we have
\begin{equation*}
\begin{array}{l l}
\displaystyle S(z) = \left\{ \begin{array}{l l}
\begin{pmatrix}
\bigO(1) & \bigO(\log (z-t_{k})) \\
\bigO(1) & \bigO(\log (z-t_{k}))
\end{pmatrix}, & \mbox{if } z \mbox{ is outside the lenses}, \\
\begin{pmatrix}
\bigO(\log (z-t_{k})) & \bigO(\log (z-t_{k})) \\
\bigO(\log (z-t_{k})) & \bigO(\log (z-t_{k}))
\end{pmatrix}, & \mbox{if } z \mbox{ is inside the lenses},
\end{array} \right. & \displaystyle \mbox{ if } \re \alpha_{k} = 0, \\[0.9cm]
		
\displaystyle S(z) = \left\{ \begin{array}{l l}
\begin{pmatrix}
\bigO(1) & \bigO(1) \\
\bigO(1) & \bigO(1)
\end{pmatrix}, & \mbox{if } z \mbox{ is outside the lenses}, \\
\begin{pmatrix}
\bigO((z-t_{k})^{-\alpha_{k}}) & \bigO(1) \\
\bigO((z-t_{k})^{-\alpha_{k}}) & \bigO(1)
\end{pmatrix}, & \mbox{if } z \mbox{ is inside the lenses},
\end{array} \right. & \displaystyle \mbox{ if } \re \alpha_{k} > 0, \\[0.9cm]
		
\displaystyle S(z) = \begin{pmatrix}
\bigO(1) & \bigO((z-t_{k})^{\alpha_{k}}) \\
\bigO(1) & \bigO((z-t_{k})^{\alpha_{k}}) 
\end{pmatrix}, & \displaystyle \mbox{ if } \re \alpha_{k} < 0.
\end{array}
\end{equation*}
\end{itemize}
Since $\gamma_{+},\gamma_{-} \subset U$ and $\re \xi(z) > 0$ for $z \in U\setminus \mathbb{T}$ (recall the discussion below \eqref{integral form of xi}), the jump matrices $S_{-}(z)^{-1}S_{+}(z)$ on $\gamma_{+}\cup \gamma_{-}$ are exponentially close to $I$ as $n \to + \infty$, and this convergence is uniform outside fixed neighborhoods of $t_{0},\dots,t_{m}$.

Our next task is to find suitable approximations (called ``parametrices") for $S$ in different regions of the complex plane. 

\subsection{Global parametrix $P^{(\infty)}$}

In this subsection, we will construct a global parametrix $P^{(\infty)}$ that is defined as the solution to the following RH problem. We will show in Subsection \ref{subsection:small norm} below that $P^{(\infty)}$ is a good approximation of $S$ outside fixed neighborhoods of $t_{0},\dots,t_{m}$.

\subsubsection*{RH problem for $P^{(\infty)}$}
\begin{itemize}
\item[(a)] $P^{(\infty)}:\mathbb{C}\setminus \mathbb{T} \to \mathbb{C}^{2\times 2}$ is analytic.

\item[(b)] The jumps are given by
\begin{align}\label{jumps of Pinf}
P^{(\infty)}_+(z) = P^{(\infty)}_-(z) \begin{pmatrix}
0 & e^{W(z)}\omega(z) \\ -e^{-W(z)}\omega(z)^{-1} & 0
\end{pmatrix}, & & z \in \mathbb{T}\setminus \{t_{0},\dots,t_{m}\}.
\end{align}

\item[(c)] As $z \to \infty$, we have $P^{(\infty)}(z) = I + \bigO(z^{-1})$. 

\item[(d)] As $z \to t_{k}$ from $|z| \lessgtr 1$, $k\in\{0,\dots,m\}$, we have $P^{(\infty)}(z) = \bigO(1)(z-t_{k})^{-(\frac{\alpha_{k}}{2}\pm \beta_{k})\sigma_{3}}$.
\end{itemize}
The unique solution to the above RH problem is given by
\begin{align}\label{P-infty-sol}
	P^{(\infty)}(z) = \begin{cases}
	D(z)^{\sigma_3}\begin{pmatrix} 0 & 1 \\ -1 & 0
	\end{pmatrix}, & \mbox{ if }  |z| < 1, \\
	D(z)^{\sigma_3}, & \mbox{ if }  |z| > 1,
	\end{cases}
\end{align}
where $D(z)$ is the Szeg\H{o} function defined by
\begin{align}
	& D(z) = D_{W}(z) \prod_{k=0}^{m}D_{\alpha_{k}}(z)D_{\beta_{k}}(z), & & D_{W}(z)=\exp{\left(\frac{1}{2\pi i}\int_{\mathbb{T}} \frac{W(s)}{s-z}ds\right)}, \label{D_w_t def} \\
	& D_{\alpha_{k}}(z) = \exp{\left(\frac{1}{2\pi i}\int_{\mathbb{T}} \frac{\log \omega_{\alpha_{k}}(s)}{s-z}ds\right)}, & & D_{\beta_{k}}(z) = \exp{\left(\frac{1}{2\pi i}\int_{\mathbb{T}} \frac{\log \omega_{\beta_{k}}(s)}{s-z}ds\right)}. \label{DalphakDbetakdef}
\end{align}
The branches of the logarithms in (\ref{DalphakDbetakdef}) can be arbitrarily chosen as long as $\log \omega_{\alpha_k}(s)$ and $\log \omega_{\beta_k}(s)$ are continuous on $\mathbb{T} \setminus t_k$. The function $D$ is analytic on $\mathbb{C}\setminus \mathbb{T}$ and satisfies  the jump condition $D_+(z) = D_-(z)e^{W(z)}\omega(z)$ on $\mathbb{T} \setminus \{t_0, \dots, t_m\}$. The expressions for $D_{\alpha_{k}}$ and $D_{\beta_{k}}$ can be simplified as in \cite[eqs. (4.9)--(4.10)]{DIK2011}; we have
\begin{align}\label{DaDb simplify}
	D_{\alpha_{k}}(z) D_{\beta_{k}}(z) = \begin{cases}
	\ds \left( \frac{z-t_{k}}{t_{k}e^{i\pi}} \right)^{\frac{\alpha_{k}}{2}+\beta_{k}} = \frac{e^{(\frac{\alpha_{k}}{2}+\beta_{k})(\log |z-t_{k}| + i \hat{\arg}_{k}(z-t_{k}))}}{e^{(\frac{\alpha_{k}}{2}+\beta_{k})(i\theta_{k}+i\pi)}}, & \mbox{if } |z|<1, \\
	\ds \left( \frac{z-t_{k}}{z} \right)^{-\frac{\alpha_{k}}{2}+\beta_{k}} = \frac{e^{(\beta_{k}-\frac{\alpha_{k}}{2})(\log |z-t_{k}| + i \hat{\arg}_{k}(z-t_{k}))}}{e^{(\beta_{k}-\frac{\alpha_{k}}{2})(\log |z| + i \hat{\arg}_{k} z)}}, & \mbox{if } |z|>1,
	\end{cases}
\end{align}
where $\hat{\arg}_{k}$ was defined below \eqref{omega beta ana cont}. Using \eqref{W as a Laurent series}, we can also simplify $D_{W}$ as
\begin{align}\label{DW simplified}
D_{W}(z) = \begin{cases}
e^{W_{0}+W_{+}(z)}, & |z|<1, \\
e^{-W_{-}(z)}, & |z|>1.
\end{cases}
\end{align}

\subsection{Local parametrices $P^{(t_k)}$}

In this subsection, we build parametrices $P^{(t_k)}(z)$ in small open disks $\mathcal{D}_{t_k}$ of $t_k$, $k=0,\dots,m$. The disks $\mathcal{D}_{t_k}$ are taken sufficiently small such that $\mathcal{D}_{t_k} \subset U$ and $\mathcal{D}_{t_k} \cap \mathcal{D}_{t_j} = \emptyset$ for $j \neq k$. Since we assume that the $t_{k}$'s remain bounded away from each other, we can (and do) choose the radii of the disks to be fixed. The parametrices $P^{(t_k)}(z)$ are defined as the solution to the following RH problem. We will show in Subsection \ref{subsection:small norm} below that $P^{(t_k)}$ is a good approximation for $S$ in $\mathcal{D}_{t_k}$.
\subsubsection*{RH problem for $P^{(t_{k})}$}
\begin{itemize}
\item[(a)] $P^{(t_k)}: \mathcal{D}_{t_{k}}\setminus(\mathbb{T} \cup \gamma_{+} \cup \gamma_{-}) \to \mathbb{C}^{2\times 2}$ is analytic.
\item[(b)] For $z\in (\mathbb{T} \cup \gamma_{+} \cup \gamma_{-})\cap \mathcal{D}_{t_{k}}$, $P_{-}^{(t_k)}(z)^{-1}P_{+}^{(t_k)}(z)=S_{-}(z)^{-1}S_{+}(z)$.
\item[(c)] As $n\to+\infty$, $P^{(t_k)}(z)=(I+\bigO(n^{-1+2|\re \beta_{k}|}))P^{(\infty)}(z)$ uniformly for $z\in\partial\mathcal{D}_{t_k}$.
\item[(d)] As $z\to t_k$, $S(z)P^{(t_k)}(z)^{-1}=\bigO(1)$.
\end{itemize}
A solution to the above RH problem can be constructed using hypergeometric functions as in \cite{DIK2011, FouMarSou}. Consider the function
\begin{align*}
f_{t_{k}}(z):= 2\pi i \int_{t_{k}}^{z} \psi(s) \frac{ds}{is}, \qquad z \in \mathcal{D}_{t_{k}},
\end{align*}
where the path is a straight line segment from $t_{k}$ to $z$. This is a conformal map from $\mathcal{D}_{t_k}$ to a neighborhood of $0$, which satisfies
\begin{equation}\label{asymptotics for f in D_t}
	f_{t_{k}}(z) = 2\pi t_{k}^{-1} \psi(t_{k}) (z-t_{k}) \left( 1+\bigO(z-t_{k}) \right), \qquad \mbox{ as } z \to t_{k}.
\end{equation}
If $\mathcal{D}_{t_{k}} \cap (-\infty,0] = \emptyset$, $f_{t_{k}}$ can also be expressed as
\begin{align*}
f_{t_{k}}(z) = - 2 \times \begin{cases}
\xi(z)-\xi_{+}(t_{k}), & |z|<1, \\
-(\xi(z)-\xi_{-}(t_{k})), & |z|>1.
\end{cases}
\end{align*}
If $\mathcal{D}_{t_{k}} \cap (-\infty,0] \neq \emptyset$, then instead we have
\begin{align}
& f_{t_{k}}(z) = - 2 \times \begin{cases}
\xi(z)-\xi_{+}(t_{k}), & |z|<1, \; \im z >0, \\
\xi(z)-\xi_{+}(t_{k})-\pi i, & |z|<1, \; \im z <0, \\
-(\xi(z)-\xi_{-}(t_{k})), & |z|>1, \; \im z >0, \\
-(\xi(z)-\xi_{-}(t_{k})+\pi i), & |z|>1, \; \im z <0,
\end{cases} & & \mbox{if } \im t_{k}>0, \label{ftk in a special case} \\
& f_{t_{k}}(z) = - 2 \times \begin{cases}
\xi(z)-\xi_{+}(t_{k})+\pi i, & |z|<1, \; \im z >0, \\
\xi(z)-\xi_{+}(t_{k}), & |z|<1, \; \im z <0, \\
-(\xi(z)-\xi_{-}(t_{k})-\pi i), & |z|>1, \; \im z >0, \\
-(\xi(z)-\xi_{-}(t_{k})), & |z|>1, \; \im z <0,
\end{cases} & & \mbox{if } \im t_{k}<0. \nonumber
\end{align}
If $t_{k}=-1$, \eqref{ftk in a special case} also holds with $\xi_{\pm}(t_k) := \lim_{\epsilon \to 0^{+}}\xi_{\pm}(e^{(\pi - \epsilon)i})=0$. We define $\omega_{k}$ and $\widetilde{W}_{k}$ by
\begin{align*}
& \omega_{k}(z)= 
e^{-2\pi i \beta_{k}\hat{\theta}(z;k)}z^{\beta_{k}}t_{k}^{-\beta_{k}} \prod_{j\neq k} \omega_{\alpha_{j}}(z)\omega_{\beta_{j}}(z), & & \widetilde{W}_{k}(z) = \check{\omega}_{\alpha_{k}}(z)^{\frac{1}{2}} \times \begin{cases}
e^{- \frac{i\pi\alpha_{k}}{2}}, & z \in Q_{+,k}^{R}\cup Q_{-,k}^{L}, \\
e^{ \frac{i\pi\alpha_{k}}{2}}, & z \in Q_{-,k}^{R} \cup Q_{+,k}^{L},
\end{cases}
\end{align*}
where $\hat{\theta}(z;k)=1$ if $\im z <0$ and $k=0$ and $\hat{\theta}(z;k)=0$ otherwise, $z^{\beta_{k}} := |z|^{\beta_{k}}e^{i\beta_{k} \arg_{0} z}$,
\begin{align}\label{sqrtomegaalphak}
& \check{\omega}_{\alpha_{k}}(z)^{1/2} := \frac{(z-t_{k})^{\frac{\alpha_{k}}{2}}}{(z t_{k} e^{i \ell_{k}(z)})^{\alpha_{k}/4}} := \frac{e^{\frac{\alpha_{k}}{2}(\log |z-t_{k}|+i \check{\arg}_{k}(z-t_{k}))}}{e^{\frac{\alpha_{k}}{4}(\log |z| + i \arg_{0}(z) + i \theta_{k} + i \ell_{k}(z))}}, 
\end{align}
and (see Figure \ref{fig: four quadrants})
\begin{align*}
	Q_{\pm,k}^{R}&= \{ z \in \mathcal{D}_{t_{k}}: \mp \re f_{t_{k}}(z) > 0 \mbox{, } \im f_{t_{k}}(z) >0 \}, \\
	Q_{\pm,k}^{L}&= \{ z \in \mathcal{D}_{t_{k}}: \mp \re f_{t_{k}}(z) > 0 \mbox{, } \im f_{t_{k}}(z) <0 \}.
\end{align*}
The argument $\check{\arg}_{k}(z-t_{k})$ in \eqref{sqrtomegaalphak} is defined to have a discontinuity for $z \in (\overline{Q_{-,k}^{L}} \cap \overline{Q_{-,k}^{R}}) \cup [z_{\star,k},t_{k}\infty)$, $z_{\star,k}:=\overline{Q_{-,k}^{L}} \cap \overline{Q_{-,k}^{R}}\cap \partial \mathcal{D}_{t_{k}}$, and such that $\check{\arg}_{k}((1-0_{+})t_{k}-t_{k})=\theta_{k}+\pi$. Note that $\check{\arg}_{k}(z-t_{k})$ is merely a small deformation of the argument $\hat{\arg}_{k}(z-t_{k})$ defined below \eqref{omega beta ana cont}. This small deformation is needed to ensure that $E_{t_{k}}$ in \eqref{E in D_t} below is analytic in $\mathcal{D}_{t_{k}}$.

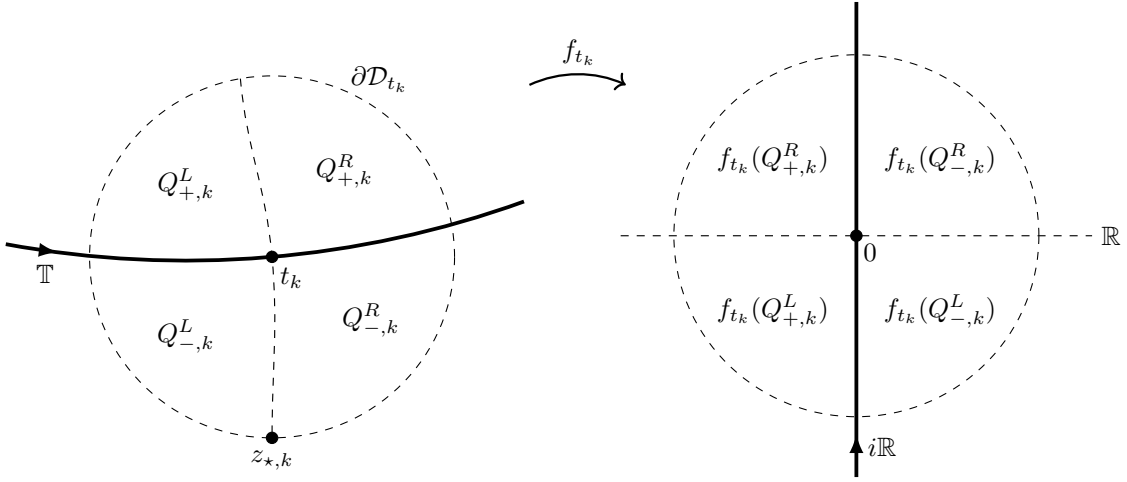
\begin{figure}
\begin{center}
\begin{tikzpicture}[scale=1]
\draw[line width=0.5 mm,->-=0.1] ([shift=(-100.5:13cm)]0,0) arc (-100.5:-70:13cm);
\node at (-84:13.3) {$t_{k}$};
\draw[fill] (-85:13) circle (0.075);

\draw[dashed] (-85:13) circle (2.4);

\draw[black, dashed, line width=0.15 mm] ($(-85:13)+(180-85+5:2.4)$) to [out=95-180, in=180-85] (-85:13) to [out=-85, in=180-85-5] ($(-85:13)+(-85-5:2.4)$);

\node at (-80:12) {$Q_{+,k}^{R}$};
\node at (-90:12) {$Q_{+,k}^{L}$};
\node at (-90:14) {$Q_{-,k}^{L}$};
\node at (-80:14) {$Q_{-,k}^{R}$};

\node at (-98:13.3) {$\mathbb{T}$};
\node at (-76.5:10.9) {$\partial\mathcal{D}_{t_k}$};

\draw[fill] ($(-85:13)+(-85-5:2.4)$) circle (0.075);
\node at ($(-85:13)+(-85-5:2.7)$) {$z_{\star,k}$};
\end{tikzpicture}
\hspace{-0.2cm}
\raisebox{0cm}{\begin{tikzpicture}[scale=1]
\draw[line width=0.3 mm,->=1] (-4.3,2) arc
    [
        start angle=115,
        end angle=65,
        x radius=1.5cm,
        y radius =1.5cm
    ] ;

\draw[black,line width=0.5 mm,->-=0.09] (0,-3.2)--(0,3.1);
\draw[black,dashed,line width=0.15 mm] (-3.1,0)--(3.1,0);

\draw[fill] (0:0) circle (0.075);

\draw[dashed] (0:0) circle (2.4);

\node at (-3.65,2.43) {$f_{t_k}$};
\node at (0.18,-0.22) {$0$};
\node at (3.35,0) {$\mathbb{R}$};
\node at (0.34,-2.8) {$i\mathbb{R}$};

\node at (-1.1,1) {$f_{t_k}(Q_{+,k}^{R})$};
\node at (-1.1,-1) {$f_{t_k}(Q_{+,k}^{L})$};
\node at (1.1,-1) {$f_{t_k}(Q_{-,k}^{L})$};
\node at (1.1,1) {$f_{t_k}(Q_{-,k}^{R})$};

\end{tikzpicture}}
\end{center}
\caption{\label{fig: four quadrants} The four quadrants $Q_{\pm,k}^{R}$, $Q_{\pm,k}^{L}$ near $t_k$ and their images under the map $f_{t_k}$. }
\end{figure}

Note that $\omega_{k}$ is analytic in $\mathcal{D}_{t_{k}}$. We now use the confluent hypergeometric model RH problem, whose solution is denoted $\Phi_{\mathrm{HG}}(z;\alpha_{k},\beta_{k})$ (see Appendix \ref{appendix:hypergeometric RHP} for the definition and properties of $\Phi_{\mathrm{HG}}$). If $k \neq 0$ and $\mathcal{D}_{t_{k}} \cap (-\infty,0] = \emptyset$, we define
\begin{equation}\label{def of local param}
P^{(t_{k})}(z) = E_{t_{k}}(z)\Phi_{\mathrm{HG}}(nf_{t_{k}}(z);\alpha_{k},\beta_{k})\widetilde{W}_{k}(z)^{-\sigma_{3}}e^{-n\xi(z)\sigma_{3}}e^{-\frac{W(z)}{2}\sigma_{3}}\omega_{k}(z)^{-\frac{\sigma_{3}}{2}},
\end{equation}
where $E_{t_{k}}$ is given by 
\begin{equation}\label{E in D_t}
	E_{t_{k}}(z) = P^{(\infty)}(z) \omega_{k}(z)^{\frac{\sigma_{3}}{2}}e^{\frac{W(z)}{2}\sigma_{3}} \widetilde{W}_{k}(z)^{\sigma_{3}}\hspace{-0.08cm} \left\{ \hspace{-0.18cm} \begin{array}{l l}
	e^{ \frac{i\pi\alpha_{k}}{4}\sigma_{3}}e^{-i\pi\beta_{k} \sigma_{3}}, \hspace{-0.2cm} & z \in Q_{+,k}^{R} \\
	e^{-\frac{i\pi\alpha_{k}}{4}\sigma_{3}}e^{-i\pi\beta_{k}\sigma_{3}}, \hspace{-0.2cm} & z \in Q_{+,k}^{L} \\
	e^{\frac{i\pi\alpha_{k}}{4}\sigma_{3}}\begin{pmatrix}
	0 & 1 \\ -1 & 0
	\end{pmatrix} , \hspace{-0.2cm} & z \in Q_{-,k}^{L} \\
	e^{-\frac{i\pi\alpha_{k}}{4}\sigma_{3}}\begin{pmatrix}
	0 & 1 \\ -1 & 0
	\end{pmatrix} , \hspace{-0.2cm} & z \in Q_{-,k}^{R} \\
	\end{array} \hspace{-0.2cm} \right\} \hspace{-0.08cm} e^{n\xi_{+}(t_{k})\sigma_{3}} (nf_{t_{k}}(z))^{\beta_{k}\sigma_{3}}.
\end{equation} 
Here the branch of $f_{t_k}(z)^{\beta_k}$ is such that $f_{t_k}(z)^{\beta_k} = |f_{t_{k}}(z)|^{\beta_k} e^{\beta_k i\arg f_{t_{k}}(z)}$ with $\arg f_{t_k}(z) \in (-\frac{\pi}{2}, \frac{3\pi}{2})$, and the branch for the square root of $\omega_{k}(z)$ can be chosen arbitrarily as long as $\omega_{k}(z)^{1/2}$ is analytic in $\mathcal{D}_{t_{k}}$ (note that $P^{(t_{k})}(z)$ is invariant under a sign change of $\omega_{k}(z)^{1/2}$). If $k \neq 0$, $\mathcal{D}_{t_{k}} \cap (-\infty,0] \neq \emptyset$ and $\im t_{k} \geq 0$ (resp. $\im t_{k} < 0$), then we define $P^{(t_{k})}(z)$ as in \eqref{def of local param} but with $\xi(z)$ replaced by $\xi(z)+\pi i \theta_{-}(z)$ (resp. $\xi(z)+\pi i \theta_{+}(z)$), where 
\begin{align*}
\theta_{-}(z):= \begin{cases}
1, & \mbox{if } \im z <0, \; |z|>1, \\
-1, & \mbox{if } \im z <0, \; |z|<1, \\
0, & \mbox{otherwise,}
\end{cases} \qquad \theta_{+}(z):= \begin{cases}
-1, & \mbox{if } \im z >0, \; |z|>1, \\
1, & \mbox{if } \im z >0, \; |z|<1, \\
0, & \mbox{otherwise.}
\end{cases}
\end{align*}
Using the definition of $\widetilde{W}_{k}$ and the jumps \eqref{jumps of Pinf} of $P^{(\infty)}$, we verify that $E_{t_{k}}$ has no jumps in $\mathcal{D}_{t_{k}}$. Moreover, since $P^{(\infty)}(z) = \bigO(1)(z-t_{k})^{-(\frac{\alpha_{k}}{2}\pm \beta_{k})\sigma_{3}}$ as $z\to t_{k}$, $\pm (1-|z|) > 0$, we infer from \eqref{E in D_t} that $E_{t_{k}}(z) = \bigO(1)$ as $z \to t_{k}$, and therefore $E_{t_{k}}$ is analytic in $\mathcal{D}_{t_{k}}$. Using \eqref{E in D_t}, we see that $E_{t_{k}}(z) = \bigO(1)n^{\beta_{k}\sigma_{3}}$ as  $n \to \infty$, uniformly for $z \in \mathcal{D}_{t_{k}}$.  Since $P^{(t_{k})}$ and $S$ have the same jumps on $(\mathbb{T}\cup \gamma_{+}\cup \gamma_{-})\cap \mathcal{D}_{t_{k}}$, $S(z)P^{(t_{k})}(z)^{-1}$ is analytic in $\mathcal{D}_{t_{k}} \setminus \{t_{k}\}$. Furthermore, by \eqref{lol 35} and condition (d) in the RH problem for $S$, as $z \to t_{k}$ from outside the lenses we have that $S(z)P^{(t_{k})}(z)^{-1}$ is $\bigO(\log(z-t_{k}))$ if $\re \alpha_{k} = 0$, is $\bigO(1)$ if $\re \alpha_{k} > 0$, and is $\bigO((z-t_{k})^{\alpha_{k}})$ if $\re \alpha_{k} < 0$. In all cases, the singularity of $S(z)P^{(t_{k})}(z)^{-1}$ at $z = t_{k}$ is removable and therefore $P^{(t_{k})}$ in \eqref{def of local param} satisfies condition (d) of the RH problem for $P^{(t_{k})}$.
	
The value of $E_{t_{k}}(t_{k})$ can be obtained by taking the limit $z \to t_{k}$ in \eqref{E in D_t} (for example from the quadrant $Q_{+,k}^{R}$). Using \eqref{P-infty-sol}, \eqref{D_w_t def}, \eqref{DaDb simplify}, \eqref{asymptotics for f in D_t} and \eqref{E in D_t}, we obtain 
\begin{align}\label{Etk at tk}
E_{t_{k}}(t_{k}) = \begin{pmatrix}
0 & 1 \\
-1 & 0
\end{pmatrix} \Lambda_{k}^{\sigma_{3}},
\end{align}
where
\begin{equation}\label{def Lambda}
	\Lambda_{k} = e^{\frac{W(t_{k})}{2}}D_{W,+}(t_{k})^{-1}\Bigg[\prod_{j \neq k}D_{\alpha_{j},+}(t_{k})^{-1}D_{\beta_{j},+}(t_{k})^{-1} \omega_{\alpha_{j}}^{\frac{1}{2}}(t_{k})\omega_{\beta_{j}}^{\frac{1}{2}}(t_{k})\Bigg](2\pi \psi(t_{k})n)^{\beta_{k}}e^{n \xi_{+}(t_{k})}.
\end{equation}
In (\ref{def Lambda}), the branch of $\omega_{\alpha_{j}}^{\frac{1}{2}}(t_{k})$ is as in (\ref{sqrtomegaalphak}) and $\omega_{\beta_{j}}^{\frac{1}{2}}(t_{k})$ is defined by
\begin{align*}
	& \omega_{\beta_{j}}^{\frac{1}{2}}(t_k) := e^{i\frac{\beta_j}{2}(\theta_{k} - \theta_{j})} \times \begin{cases}
e^{\frac{i \pi}{2} \beta_{j}}, & \mbox{if } 0 \leq \theta_{k} < \theta_{j}, \\
e^{-\frac{i \pi}{2} \beta_{j}}, & \mbox{if } \theta_{j} \leq \theta_{k} < 2 \pi.
\end{cases} 
\end{align*}
The expression for $\Lambda_k$ can be further simplified as follows. A simple computation shows that
\begin{align*}
D_{\alpha_{j},+}(z) & = |z-t_{j}|^{\frac{\alpha_{j}}{2}}\exp \left( \frac{i \alpha_{j}}{2}\left[ \hat{\arg}_{j}(z-t_{j}) - \theta_{j} - \pi \right] \right) 
	\\	
& = |z-t_{j}|^{\frac{\alpha_{j}}{2}}\exp \left( \frac{i \alpha_{j}}{2}\left[ \frac{\ell_{j}(z)}{2}+ \frac{\arg_{0} z - \theta_{j}}{2} - \pi \right] \right), 
	\\
D_{\beta_{j},+}(z) & = |z-t_{j}|^{\beta_{j}}\exp \left( i \beta_{j} \left[ \hat{\arg}_{j}(z-t_{j}) - \theta_{j} - \pi \right] \right) 
	\\
& = |z-t_{j}|^{\beta_{j}}\exp \left( i \beta_{j}\left[ \frac{\ell_{j}(z)}{2}+ \frac{\arg_{0} z - \theta_{j}}{2} - \pi \right] \right)
\end{align*}
for $z \in \mathbb{T}$. Therefore, the product in brackets in \eqref{def Lambda} can be rewritten as 
\begin{align*}
	\prod_{j \neq k} |t_{k}-t_{j}|^{-\beta_{j}}\exp \left( - \frac{i\alpha_{j}}{2} \frac{\theta_{k} - \theta_{j}}{2}  \right)\prod_{j=0}^{k-1} \exp \left( \frac{\pi i \alpha_{j}}{4} \right)  \prod_{j=k+1}^{m} \exp \left( -\frac{\pi i \alpha_{j}}{4} \right),
\end{align*}
and thus
\begin{align*}
	\Lambda_{k} = e^{\frac{W(t_{k})}{2}}D_{W,+}(t_{k})^{-1}e^{\frac{i \lambda_{k}}{2}}(2\pi \psi(t_{k})n)^{\beta_{k}}\prod_{j \neq k} |t_{k}-t_{j}|^{-\beta_{j}},
\end{align*}
where
\begin{equation}
\lambda_{k} = \sum_{j=0}^{k-1} \frac{\pi \alpha_{j}}{2} - \sum_{j=k+1}^{m} \frac{\pi \alpha_{j}}{2} - \sum_{j \neq k} \frac{\alpha_{j}(\theta_{k}-\theta_{j})}{2} + 2\pi n \int_{t_{k}}^{-1} \psi(s)\frac{ds}{is}.
\end{equation}
Using \eqref{def of local param} and \eqref{Asymptotics HG}, we obtain 
\begin{equation}\label{asymptotics on the disk D_t}
P^{(t_{k})}(z)P^{(\infty)}(z)^{-1} = I + \frac{\beta_{k}^{2}-\frac{\alpha_{k}^{2}}{4}}{n f_{t_{k}}(z)} E_{t_{k}}(z) \begin{pmatrix}
-1 & \tau(\alpha_{k},\beta_{k}) \\ - \tau(\alpha_{k},-\beta_{k}) & 1
\end{pmatrix}E_{t_{k}}(z)^{-1} + \bigO (n^{-2+2|\re \beta_{k}|}),
\end{equation}
as $n \to \infty$ uniformly for $z \in \partial \mathcal{D}_{t_{k}}$, where $\tau(\alpha_{k},\beta_{k})$ is defined in \eqref{def of tau}.

\subsection{Small norm RH problem}\label{subsection:small norm}
We consider the function $R$ defined by
\begin{align}
\label{R-function}
R(z) = \begin{cases}
S(z)P^{(\infty)}(z)^{-1}, & z \in \C \setminus (\cup_{k=0}^{m}\overline{\mathcal{D}_{t_k}} \cup \mathbb{T} \cup  \gamma_{+} \cup \gamma_{-}),  \\
S(z)P^{(t_{k})}(z)^{-1}, & z \in \mathcal{D}_{t_{k}}\setminus (\mathbb{T} \cup  \gamma_{+} \cup \gamma_{-}), \; k=0,\dots,m.
\end{cases}
\end{align}
We have shown in the previous section that $P^{(t_{k})}$ and $S$ have the same jumps on $\mathbb{T} \cup  \gamma_{+} \cup \gamma_{-}$ and that $S(z)P^{(t_{k})}(z)^{-1}=\bigO(1)$ as $z\to t_{k}$. Hence $R$ is analytic in $\cup_{k=0}^m\mathcal{D}_{t_k}$. Using also the RH problems for $S$, $P^{(\infty)}$ and $P^{(t_{k})}$, we conclude that $R$ satisfies the following RH problem.
\subsubsection*{RH problem for $R$}
\begin{itemize}
\item[(a)] $R : \C \setminus \Gamma_{R} \to \mathbb{C}^{2 \times 2}$ is analytic, where $\Gamma_{R} = \cup_{k=0}^{m} \partial \mathcal{D}_{t_{k}} \cup \big((\gamma_{+} \cup \gamma_{-}) \setminus \cup_{k=0}^{m} \mathcal{D}_{t_{k}}\big)$ and the circles $\partial \mathcal{D}_{t_{k}}$ are oriented in the clockwise direction.
\item[(b)] The jumps are given by
\begin{align*}
& R_{+}(z) = R_{-}(z) P^{(\infty)}(z)\begin{pmatrix}
1 & 0 \\ e^{-W(z)}\omega(z)^{-1}e^{-2n\xi(z)} & 1
\end{pmatrix}P^{(\infty)}(z)^{-1}, & & z \in (\gamma_{+} \cup \gamma_{-}) \setminus \cup_{k=0}^{m} \overline{\mathcal{D}_{t_{k}}}, \\
& R_{+}(z) = R_{-}(z) P^{(t_{k})}(z)P^{(\infty)}(z)^{-1}, & & z \in \partial \mathcal{D}_{t_{k}}, \, k=0,\dots,m.
\end{align*}
\item[(c)] As $z \to \infty$, $R(z) = I + \bigO(z^{-1})$. 
\item[(d)]  
As $z \to z^{*}\in \Gamma_{R}^{*}$, where $\Gamma_{R}^{*}$ is the set of self-intersecting points of $\Gamma_{R}$, we have $R(z) = \bigO(1)$.
\end{itemize}
Recall that $\re \xi(z) \geq c > 0$ for $z \in (\gamma_{+} \cup \gamma_{-}) \setminus \cup_{k=0}^{m} \mathcal{D}_{t_{k}}$. Moreover, we see from \eqref{P-infty-sol} that $P^{(\infty)}(z)$ is bounded for $z$ away from the points $t_{0},\dots,t_{m}$. Using also \eqref{asymptotics on the disk D_t}, we conclude that as $n \to + \infty$
\begin{align}
& J_{R}(z) = I + \bigO(e^{-cn}), & & \mbox{uniformly for } z \in (\gamma_{+} \cup \gamma_{-}) \setminus \cup_{k=0}^{m} \overline{\mathcal{D}_{t_{k}}}, \label{JR est 1} \\
& J_{R}(z) = I + J_{R}^{(1)}(z)n^{-1} + \bigO(n^{-2+2\beta_{\max}}), & & \mbox{uniformly for } z \in \cup_{k=0}^{m} \partial \mathcal{D}_{t_{k}}, \label{JR est 2}
\end{align}
where $J_{R}(z):=R_{-}^{-1}(z)R_{+}(z)$ and
\begin{align*}
& J_{R}^{(1)}(z) = \frac{\beta_{k}^{2}-\frac{\alpha_{k}^{2}}{4}}{f_{t_{k}}(z)} E_{t_{k}}(z) \begin{pmatrix}
-1 & \tau(\alpha_{k},\beta_{k}) \\ - \tau(\alpha_{k},-\beta_{k}) & 1
\end{pmatrix}E_{t_{k}}(z)^{-1}, & & z \in \partial\mathcal{D}_{t_{k}}.
\end{align*}
Furthermore, it is easy to see that the $\bigO$-terms in \eqref{JR est 1}--\eqref{JR est 2} are uniform for $(\theta_{1},\dots,\theta_{m})$ in any given compact subset $\Theta \subset (0,2\pi)_{\mathrm{ord}}^{m}$, for $\alpha_{0},\dots,\alpha_{m}$ in any given compact subset $\mathfrak{A}\subset \{z \in \mathbb{C}: \re z >-1\}$, and for $\beta_{0},\dots,\beta_{m}$ in any given compact subset $\mathfrak{B}\subset\{z \in \mathbb{C}: \re z \in (-\frac{1}{2},\frac{1}{2})\}$. Therefore, $R$ satisfies a small norm RH problem, and the existence of $R$ for all sufficiently large $n$ can be proved using standard theory \cite{DZ1993, Deiftetal} as follows. Define the operator $\mathcal{C}:L^{2}(\Gamma_{R})\to L^{2}(\Gamma_{R})$ by $\mathcal{C}f(z) = \frac{1}{2\pi i}\int_{\Gamma_{R}}\frac{f(s)}{s-z}dz$, and denote $\mathcal{C}_{+}f$ and $\mathcal{C}_{-}f$ for the left and right non-tangential limits of $\mathcal{C}f$. Since $\Gamma_{R}$ is a compact set, by \eqref{JR est 1}--\eqref{JR est 2} we have $J_{R}-I \in L^{2}(\Gamma_{R})\cap L^{\infty}(\Gamma_{R})$, and we can define
\begin{align*}
\mathcal{C}_{J_{R}}: L^{2}(\Gamma_{R})+L^{\infty}(\Gamma_{R}) \to L^{2}(\Gamma_{R}), \qquad \mathcal{C}_{J_{R}}f=\mathcal{C}_{-}(f(J_{R}-I)), \qquad f \in L^{2}(\Gamma_{R})+L^{\infty}(\Gamma_{R}).
\end{align*}
Using $\| \mathcal{C}_{J_{R}} \|_{L^{2}(\Gamma_{R}) \to L^{2}(\Gamma_{R})} \leq C \|J_{R}-I\|_{L^{\infty}(\Gamma_{R})}$ and \eqref{JR est 1}--\eqref{JR est 2}, we infer that there exists $n_{0}=n_{0}(\Theta,\mathfrak{A},\mathfrak{B})$ such that $\| \mathcal{C}_{J_{R}} \|_{L^{2}(\Gamma_{R}) \to L^{2}(\Gamma_{R})} <1$ for all $n \geq n_{0}$, all $(\theta_{1},\dots,\theta_{m})\in \Theta$, all $\alpha_{0},\dots,\alpha_{m} \in \mathfrak{A}$ and all $\beta_{0},\dots,\beta_{m}\in \mathfrak{B}$. Hence, for $n \geq n_{0}$, $I-\mathcal{C}_{J_{R}}:L^{2}(\Gamma_{R}) \to L^{2}(\Gamma_{R})$ can be inverted as a Neumann series and thus $R$ exists and is given by
\begin{align}\label{R exists}
R=I+\mathcal{C}(\mu_{R}(J_{R}-I)), \qquad \mbox{where } \quad \mu_{R}:= I + (I-\mathcal{C}_{J_{R}})^{-1}\mathcal{C}_{J_{R}}(I).
\end{align}
Using \eqref{R exists}, \eqref{JR est 1} and \eqref{JR est 2}, we obtain
\begin{align}\label{R asymp}
R(z) = I+R^{(1)}(z)n^{-1} + \bigO(n^{-2+2\beta_{\max}}), \qquad \mbox{as } n \to + \infty,
\end{align}
uniformly for $(\theta_{1},\dots,\theta_{m})\in \Theta$, $\alpha_{0},\dots,\alpha_{m} \in \mathfrak{A}$ and $\beta_{0},\dots,\beta_{m}\in \mathfrak{B}$, where $R^{(1)}$ is given by
\begin{align*}
	R^{(1)}(z) = \sum_{k=0}^{m}\frac{1}{2\pi i} \int_{\partial \mathcal{D}_{t_{k}}} \frac{J_{R}^{(1)}(s)}{s-z}ds.
\end{align*}
Since the jumps $J_{R}$ are analytic in a neighborhood of $\Gamma_{R}$, the expansion \eqref{R asymp} holds uniformly for $z \in \mathbb{C}\setminus \Gamma_{R}$. It also follows from \eqref{R exists} that \eqref{R asymp} can be differentiated with respect to $z$ without increasing the error term. For $z \in \mathbb{C}\setminus  \cup_{k=0}^{m} \mathcal{D}_{t_{k}}$, a residue calculation using \eqref{asymptotics for f in D_t}, \eqref{Etk at tk} and \eqref{asymptotics on the disk D_t} shows that (recall that $\partial \mathcal{D}_{t_{k}}$ is oriented in the clockwise direction)
\begin{align}\label{R^1}
R^{(1)}(z) = \sum_{k=0}^{m} \frac{1}{z-t_{k}}  \frac{(\beta_{k}^{2}-\frac{\alpha_{k}^{2}}{4})t_{k}}{2\pi \psi(t_{k})} \begin{pmatrix}
1 & \Lambda_{k}^{-2} \tau(\alpha_{k},-\beta_{k}) \\
-\Lambda_{k}^{2} \tau(\alpha_{k},\beta_{k}) & -1
\end{pmatrix}.
\end{align}

\begin{remark}\label{remark:s and t uniform}
Above, we have discussed the uniformity of \eqref{JR est 1}--\eqref{JR est 2} and \eqref{R asymp} in the parameters $\theta_{k},\alpha_{k},\beta_{k}$. In Section \ref{section: int in V}, we will also need the following fact, which can be proved via a direct analysis (we omit the details here, see e.g. \cite[Lemma 4.35]{BerWebbWong} for a similar situation): If $V$ is replaced by $sV$, then \eqref{JR est 1}--\eqref{JR est 2} and \eqref{R asymp} also hold uniformly for $s \in [0,1]$.
\end{remark}

\begin{remark}\label{remark:diff}
If $k_{0},\dots,k_{2m+1}\in \mathbb{N}$, $k_{0}+\dots+k_{2m+1}\geq 1$ and $\partial^{\vec{k}}:=\partial_{\alpha_{0}}^{k_{0}}\dots\partial_{\alpha_{m}}^{k_{m}}\partial_{\beta_{0}}^{k_{m+1}}\dots\partial_{\beta_{m}}^{k_{2m+1}}$, then by \eqref{P-infty-sol} we have
\begin{align*}
& \partial^{\vec{k}}J_{R}(z) = \bigO(e^{-cn}), & & \mbox{uniformly for } z \in (\gamma_{+} \cup \gamma_{-}) \setminus \cup_{k=0}^{m} \overline{\mathcal{D}_{t_{k}}},  
\end{align*}
and by the same type of arguments that led to \eqref{asymptotics on the disk D_t} we have
\begin{align*}
& \partial^{\vec{k}}J_{R}(z) = \partial^{\vec{k}}(J_{R}^{(1)}(z))n^{-1} + \bigO\bigg(\frac{(\log n)^{k_{m+1}+\dots+k_{2m+1}}}{n^{2-2\beta_{\max}}}\bigg), & & \mbox{uniformly for } z \in \cup_{k=0}^{m} \partial \mathcal{D}_{t_{k}}.
\end{align*}
It follows that
\begin{align*}
\partial^{\vec{k}}R(z) = \partial^{\vec{k}}(R^{(1)}(z))n^{-1} + \bigO\bigg(\frac{(\log n)^{k_{m+1}+\dots+k_{2m+1}}}{n^{2-2\beta_{\max}}}\bigg), \qquad \mbox{as } n \to + \infty.
\end{align*}
\end{remark}
If $W$ is replaced by $tW$, $t \in [0,1]$, then the asymptotics \eqref{JR est 1}, \eqref{JR est 2} and \eqref{R asymp} are uniform with respect to $t$ and can also be differentiated any number of times with respect to $t$ without worsening the error term.

\section{Integration in $V$}\label{section: int in V}

Our strategy is inspired by \cite{BerWebbWong} and considers a linear deformation in the potential (in \cite{BerWebbWong} the authors study Hankel determinants related to point processes on the real line, see also \cite{Charlier, ChGha, CFWW2021} for subsequent works using similar deformation techniques). Consider the potential $\hat{V}_{s}:=sV$, where $s \in [0,1]$. It is immediate to verify that
\begin{align}
2 \int_{0}^{2\pi} \log |z-e^{i\theta}| d\mu_{\hat{V}_{0}}(e^{i\theta}) = \hat{V}_{0}(z) - \ell_{0}, & & \mbox{ for } z \in \mathbb{T}, \label{var equality at s=0}
\end{align}
with $d\mu_{\hat{V}_{0}}(e^{i\theta}):=\frac{1}{2\pi}d\theta$ and $\ell_{0}=0$. Using a linear combination of \eqref{var equality at s=0} and \eqref{var equality} (writing $\hat{V}_{s}=(1-s)\hat{V}_{0}+sV$), we infer that
 \begin{align}
2 \int_{0}^{2\pi} \log |z-e^{i\theta}| d\mu_{\hat{V}_{s}}(e^{i\theta}) = \hat{V}_{s}(z) - \ell_{s}, & & \mbox{ for } z \in \mathbb{T}, 
\end{align}
holds for each $s \in [0,1]$ with $\ell_{s}:=s \ell$ and $d\mu_{\hat{V}_{s}}(e^{i\theta})=\psi_{s}(e^{i\theta})d\theta$, $\psi_{s}(e^{i\theta}):=\frac{1-s}{2\pi}+s\psi(e^{i\theta})$. In particular, this shows that $\psi_{s}(e^{i\theta})>0$ for all $s \in [0,1]$ and all $\theta \in [0,2\pi)$. Hence, we can (and will) use the analysis of Section \ref{section:steepest descent analysis} with $V$ replaced by $\hat{V}_{s}$.

We first recall the following result, which will be used for our proof.

\begin{theorem}\label{thm:V=0}[Taken from \cite{Ehr, DIK2011}]
Let $m \in \mathbb{N}$, and let $t_{k}=e^{i\theta_{k}}$, $\alpha_{k}$ and $\beta_{k}$ be such that
\begin{equation*}
0=\theta_{0} < \theta_{1} < \dots < \theta_{m} < 2\pi, \quad \mbox{ and } \quad \re \alpha_{k} > -1, \quad \re \beta_{k} \in (-\tfrac{1}{2},\tfrac{1}{2}) \quad \mbox{ for } k=0,\dots,m.
\end{equation*} 
Let $W: \mathbb{T}\to\mathbb{R}$ be analytic, and define $W_{+}$ and $W_{-}$ as in \eqref{W as a Laurent series}. As $n \to +\infty$, we have
\begin{align}
D_n(\vec\alpha,\vec\beta,0,W)= \exp \bigg( D_{2}n + D_{3} \log n + D_{4} + \bigO\bigg( \frac{1}{n^{1- 2\beta_{\max}}} \bigg)\bigg),
\end{align}
where
\begin{align*}
D_{2} = &\; W_{0}, \\
 D_{3} = &\;\sum_{k=0}^{m} \bigg(\frac{\alpha_{k}^{2}}{4}-\beta_{k}^{2}\bigg), \\
 D_{4} = &\;\sum_{\ell = 1}^{+\infty} \ell W_{\ell}W_{-\ell} + \sum_{k=0}^{m} \bigg( \beta_{k}-\frac{\alpha_{k}}{2} \bigg) W_{+}(t_{k}) - \sum_{k=0}^{m} \bigg( \beta_{k}+\frac{\alpha_{k}}{2} \bigg) W_{-}(t_{k}) \\
& + \sum_{0 \leq j < k \leq m} \bigg\{ \frac{\alpha_{j} i \beta_{k} - \alpha_{k} i \beta_{j}}{2}(\theta_{k}-\theta_{j }-\pi) + \bigg( 2\beta_{j}\beta_{k}-\frac{\alpha_{j}\alpha_{k}}{2} \bigg) \log |t_{j}-t_{k}| \bigg\} \\
& + \sum_{k=0}^{m} \log \frac{G(1+\frac{\alpha_{k}}{2}+\beta_{k})G(1+\frac{\alpha_{k}}{2}-\beta_{k})}{G(1+\alpha_{k})},
\end{align*}
where $G$ is Barnes' $G$-function. Furthermore, the above asymptotics are uniform for all $\alpha_{k}$ in compact subsets of $\{z \in \mathbb{C}: \re z >-1\}$, for all $\beta_{k}$ in compact subsets of $\{z  \in \mathbb{C}: \re z \in (-\frac{1}{2},\frac{1}{2})\}$, and for all $(\theta_{1},\dots,\theta_{m})$ in compact subsets of $(0,2\pi)_{\mathrm{ord}}^{m}$.
\end{theorem}
\begin{remark}
The above theorem, but with the $\bigO$-term replaced by $o(1)$, was proved by Ehrhardt in \cite{Ehr}. The stronger estimate $\bigO( n^{-1+ 2\beta_{\max}} )$ was obtained in \cite[Remark 1.4]{DeiftItsKrasovsky}. (In fact the results \cite{Ehr, DeiftItsKrasovsky} are valid for more general values of the $\beta_{k}$'s, but this will not be needed for us.)
\end{remark}

\begin{lemma}\label{lemma: dashint ids}
For $z\in \mathbb{T}$, we have
\begin{align}
\frac{1}{i\pi}\dashint_{\mathbb{T}}\frac{V'(w)}{w-z}dw&=\frac{1}{z}\left(1-2\pi\psi(z)\right), \label{nice identity} \\
\frac{1}{i\pi}\dashint_{\mathbb{T}}\frac{V(w)}{w-z}dw&= V_{0} + V_{+}(z) - V_{-}(z) = V_{0} + 2i \, \im(V_{+}(z)), \label{nice identity 2}
\end{align}
where $\dashint$ stands for principal value integral.
\end{lemma}
\begin{proof}
The first identity \eqref{nice identity} can be proved by a direct residue calculation using \eqref{V as a Laurent series} and \eqref{psi solved explicitly}. We give here another proof, more in the spirit of \cite[Lemma 5.8]{BerWebbWong} and \cite[Lemma 8.1]{ChGha}. 
Let $H,\varphi:\mathbb{C}\setminus \mathbb{T}\to\mathbb{C}$ be functions given by
\begin{align}\label{def of H and theta}
& H(z) = \varphi(z)\left(g'(z)-\frac{1}{2z}\right)-\frac{1}{2z}+\frac{1}{2\pi i}\int_{\mathbb{T}}\frac{V'(w)}{w-z}dw, & & \varphi(z)=\begin{cases}
-1, & |z|<1, \\
1, & |z|>1.
\end{cases}
\end{align}
Clearly, $H(\infty)=0$, and for $z\in \mathbb{T}$ we have
\begin{align*}
    H_+(z)-H_-(z)=-\left(g'_+(z)+g_-'(z)-\frac{1}{z}\right)+V'(z)=0,
\end{align*}
where for the last equality we have used \eqref{EL= in terms of g}. So $H(z)\equiv0$ by Liouville's theorem.  The identity \eqref{nice identity} now follows from the relations \eqref{uni: gprime} and 
\begin{align*}
    0=H_+(z)+H_-(z)=-(g'_+(z)-g'_-(z))-\frac{1}{z}+\frac{1}{i\pi}\dashint_{\mathbb{T}}\frac{V'(w)}{w-z}dz, \qquad z\in \mathbb{T}.
\end{align*}
The second identity \eqref{nice identity 2} follows from a direct residue computation, using \eqref{V as a Laurent series}. 

\end{proof}

\begin{proposition}\label{prop: int V}
As $n\to+\infty$,
\begin{align}
\log\frac{D_{n}(\vec\alpha,\vec\beta,V,W)}{D_{n}(\vec\alpha,\vec\beta,0,W)}=c_1 n^2 + c_2 n + c_3 + \bigO(n^{-1+2\beta_{\max}}),
\end{align}
where
\begin{align*}
c_1 & = -\frac{V_{0}}{2}-\frac{1}{2}\int_0^{2\pi}V(e^{i\theta}) d\mu_{V}(e^{i\theta}), \\
c_2&= \sum_{k=0}^{m} \frac{\alpha_{k}}{2}(V(t_{k})-V_{0}) - \sum_{k=0}^{m} 2i\beta_{k} \im(V_{+}(t_{k})) + \int_0^{2\pi}W(e^{i\theta})d\mu_V(e^{i\theta}) -W_{0}, \\
c_3&=\sum_{k=0}^m\frac{\beta_{k}^{2}-\frac{\alpha_{k}^{2}}{4}}{\psi(t_k)}\left(\frac{1}{2\pi}-\psi(t_k)\right).
\end{align*}
\end{proposition}

\begin{proof}
We will use \eqref{uni: diff id} with $V=\hat{V}_{s}$ and $\gamma=s$, i.e.
\begin{align}\label{int: V}
\partial_{s} \log D_{n}(\vec{\alpha},\vec{\beta},\hat{V}_{s},W) = \frac{1}{2\pi}\int_{0}^{2\pi}[Y^{-1}(z)Y'(z)]_{21}z^{-n+1}\partial_{s}f(z)d\theta,
\end{align}
where $f(z)=e^{-n\hat{V}_{s}(z)}\omega(z)$ and $Y(\cdot) = Y_{n}(\cdot;\vec{\alpha},\vec{\beta},\hat{V}_{s},W)$. Recall from Proposition \ref{prop: diff id} that \eqref{int: V} is valid only when $D_{k}^{(n)}(f) \neq 0$, $k=n-1,n,n+1$. However, it follows from the analysis of Subsection \ref{subsection:small norm} (see also Remark \ref{remark:s and t uniform}) that the right-hand side of \eqref{int: V} exists for all $n$ sufficiently large, for all $(\theta_{1},\dots,\theta_{m})\in \Theta$, all $\alpha_{0},\dots,\alpha_{m}\in \mathfrak{A}$, all $\beta_{0},\dots,\beta_{m}\in \mathfrak{B}$, and all $s \in [0,1]$. Hence we can extend \eqref{int: V} by continuity (see also \cite{Krasovsky, ItsKrasovsky, DeiftItsKrasovsky, Charlier, ChGha} for similar situations with more details provided). By \eqref{Y jump}, for $z\in \mathbb{T}\setminus \{t_{0},\dots,t_{m}\}$ we have
\begin{align}
& [Y(z)^{-1}Y'(z)]_{21,+}=[Y(z)^{-1}Y'(z)]_{21,-}, \\
& [Y(z)^{-1}Y'(z)]_{21}=-\frac{z^{n}}{f(z)}\left([Y(z)^{-1}Y'(z)]_{11,+}-[Y(z)^{-1}Y'(z)]_{11,-}\right),
\end{align}
and thus, using that $\partial_s\log f(z) = -nV(z)$ is analytic in a neighborhood of $\mathbb{T}$, 
\begin{align}\label{diff id in s}
\partial_{s} \log D_{n}(\vec{\alpha},\vec{\beta},\hat{V}_{s},W)=\frac{-1}{2\pi i}\int_{\mathcal{C}_e\cup\mathcal{C}_i}\left[Y^{-1}(z)Y'(z)\right]_{11}\partial_s\log f(z)dz,
\end{align}
where $\mathcal{C}_i \subset \{z:|z|<1\}\cap U$ is a closed curve oriented counterclockwise and surrounding $0$, and  $\mathcal{C}_e \subset \{z:|z|>1\} \cap U$ is a closed curve oriented clockwise and surrounding $0$. We choose $\mathcal{C}_i$ and $\mathcal{C}_e$ such that they do not intersect $\mathbb{T}\cup\gamma_+\cup \gamma_{-}\cup\mathcal{D}_{t_0}\cup\cdots\cup\mathcal{D}_{t_m}$.

Inverting the transformations $Y \mapsto T \mapsto S \mapsto R$ of Section \ref{section:steepest descent analysis} using \eqref{def of T}, \eqref{T to S} and \eqref{R-function}, for $z \in \mathcal{C}_e\cup\mathcal{C}_i$ we find
\begin{align*}
    \left[Y^{-1}(z)Y'(z)\right]_{11}=ng'(z)+\left[P^{(\infty)}(z)^{-1}P^{(\infty)\prime}(z)\right]_{11}+\left[P^{(\infty)}(z)^{-1}R(z)^{-1}R'(z)P^{(\infty)}(z)\right]_{11}.
\end{align*}
Substituting the above in \eqref{diff id in s}, we find the following exact identity:
\begin{align*}
\partial_{s} \log D_{n}(\vec{\alpha},\vec{\beta},\hat{V}_{s},W) = I_{1,s}+I_{2,s}+I_{3,s},
\end{align*}
where
\begin{align}
I_{1,s}&=\frac{-n}{2\pi i}\int_{\mathcal{C}_e\cup\mathcal{C}_i}g'(z)\partial_s\log f(z)dz, \label{I1s} \\
I_{2,s}&=\frac{-1}{2\pi i}\int_{\mathcal{C}_e\cup\mathcal{C}_i}\left[P^{(\infty)}(z)^{-1}P^{(\infty)\prime}(z)\right]_{11}\partial_s \log f(z)dz, \label{I2s} \\
I_{3,s}&=\frac{-1}{2\pi i}\int_{\mathcal{C}_e\cup\mathcal{C}_i}\left[P^{(\infty)}(z)^{-1}R(z)^{-1}R'(z)P^{(\infty)}(z)\right]_{11}\partial_s \log f(z)dz. \label{I3s}
\end{align}
Using $\partial_s\log f(z) = -nV(z)$ and \eqref{uni: gprime} (with $\psi$ replaced by $\psi_{s}$), we find
\begin{align*}
I_{1,s}=\frac{n^2}{2\pi i}\int_{\mathbb{T}}(g_+'(z)-g_-'(z))V(z)dz=-n^2\int_{\mathbb{T}}V(z)\psi_{s}(z)\frac{dz}{iz},
\end{align*}
and since $\psi_{s}=\frac{1-s}{2\pi}+s\psi$,
\begin{align}
\int_0^1 I_{1,s}ds=-\frac{n^2}{2}\int_0^{2\pi}V(e^{i\theta})\Big( \frac{1}{2\pi}+\psi(e^{i\theta}) \Big)d\theta = -\frac{n^2}{2}\bigg(V_{0}+\int_0^{2\pi}V(e^{i\theta}) d\mu_{V}(e^{i\theta}) \bigg) = c_1 n^2.
\end{align}
Now we turn to the analysis of $I_{2,s}$. Using \eqref{P-infty-sol}, we obtain
\begin{align}\label{PInfInv PPrime}
\Big[P^{(\infty)}(z)^{-1}P^{(\infty)\prime}(z)\Big]_{11}=\varphi(z)\partial_z[\log D(z)]
\end{align}
where $\varphi$ is defined in \eqref{def of H and theta}. Also, by \eqref{D_w_t def}, \eqref{DaDb simplify} and \eqref{DW simplified}, we have
\begin{align}\label{d log D}
\partial_z\log D(z)=
\begin{cases}
W_{+}'(z)+\sum_{k=0}^{m}\left(\beta_k+\frac{\alpha_k}{2}\right)\frac{1}{z-t_k}, & |z|<1, \\
-W_{-}'(z)+\sum_{k=0}^{m}\left(\beta_k-\frac{\alpha_k}{2}\right)\big(\frac{1}{z-t_k}-\frac{1}{z}\big), & |z|>1,
\end{cases}
\end{align}
where $W_\pm$ are defined in \eqref{W as a Laurent series}, and by \eqref{psi solved explicitly}, we have
\begin{align}\label{contrib WV}
- \sum_{k=-\infty}^{+\infty}|k|W_{k}V_{-k} = \int_0^{2\pi}W(e^{i\theta})d\mu_V(e^{i\theta}) -W_{0}.
\end{align}
Substituting \eqref{PInfInv PPrime} and \eqref{d log D} in \eqref{I2s}, and doing a residue computation, we obtain
\begin{align*}
I_{2,s} & = - n \sum_{k=-\infty}^{+\infty}|k|W_{k}V_{-k} + n\sum_{k=0}^{m} \frac{\alpha_{k}}{2}(V(t_{k})-V_{0}) - n \sum_{k=0}^{m} \beta_{k} \bigg( \frac{1}{\pi i}\dashint_{\mathbb{T}} \frac{V(z)}{z-t_{k}} dz - V_{0} \bigg) = c_{2}n,
\end{align*}
where for the last equality we have used \eqref{nice identity 2} and \eqref{contrib WV}. Clearly, $I_{2,s}$ is independent of $s$, and therefore $\int_{0}^{1}I_{2,s}ds = c_{2}n$. We now analyze $I_{3,s}$ as $n \to + \infty$. From \eqref{R asymp}, we have
\begin{align*}
R^{-1}(z)R'(z)=n^{-1}R^{(1)\prime}(z) + \bigO(n^{-2 + 2 \beta_{\max}}),
\end{align*}
and, using first \eqref{P-infty-sol} and then \eqref{R^1},
\begin{align*}
\left[P^{(\infty)}(z)^{-1}n^{-1}R^{(1)\prime}(z)P^{(\infty)}(z)\right]_{11}&=\frac{1}{n} \times \begin{cases}
\left[R^{(1)\prime}(z)\right]_{22}, & |z|<1 \\
\left[R^{(1)\prime}(z)\right]_{11}, & |z|>1
\end{cases} =\frac{-\varphi(z)}{2\pi n}\sum_{k=0}^{m}\frac{(\beta_k^2-\frac{\alpha_k^2}{4})t_k}{\psi(t_k)(z-t_k)^2}.
\end{align*}
Therefore, as $n \to + \infty$
\begin{align*}
I_{3,s}=\frac{1}{2\pi}\sum_{k=0}^m\frac{(\beta_{k}^{2}-\frac{\alpha_{k}^{2}}{4})t_k}{\psi(t_k)}\frac{1}{ 2\pi i}\left(\int_{C_i}-\int_{C_e}\right)\frac{V(z)}{(z-t_k)^2}dz + \bigO(n^{-1+2\beta_{\max}}).
\end{align*}
Partial integration yields
\begin{align*}
\frac{1}{2\pi i}\left(\int_{C_i}-\int_{C_e}\right)\frac{V(z)}{(z-t_k)^2}dz=\frac{1}{2\pi i}\left(\int_{C_i}-\int_{C_e}\right)\frac{V'(z)}{z-t_k}dz=\frac{1}{\pi i}\dashint_{\mathbb{T}}\frac{V'(z)}{z-t_k}dz,
\end{align*}
and thus, by \eqref{nice identity}, we have
\begin{align*}
I_{3,s} & = \frac{1}{2\pi}\sum_{k=0}^{m}\frac{(\beta_{k}^{2}-\frac{\alpha_{k}^{2}}{4})t_k}{\psi(t_k)} \frac{1}{t_k}\left(1-2\pi\psi(t_k)\right) + \bigO(n^{-1+2\beta_{\max}}), \qquad \mbox{as } n \to + \infty.
\end{align*}
Since the above asymptotics are uniform for $s \in [0,1]$ (see Remark \ref{remark:s and t uniform}), the claim follows.
\end{proof}

Theorem \ref{theorem U} now directly follows by combining Proposition \ref{prop: int V} with Theorem \ref{thm:V=0}. (The estimate \eqref{der of error in thm} follows from Remark \ref{remark:diff}.)

\section{Proofs of Corollaries \ref{coro:smooth}, \ref{coro:log}, \ref{coro:counting}, \ref{coro:order}, \ref{coro:rigidity}}\label{Section: rigidity}
Let $e^{\phi_{1}},\dots,e^{i\phi_{n}}$ be distributed according to \eqref{point process} with $\phi_{1},\dots,\phi_{n}\in [0,2\pi)$. Recall that $N_{n}(\theta) = \#\{\phi_{j}\in [0,\theta)\}$ and that the angles $\phi_{1},\dots,\phi_{n}$ arranged in increasing order are denoted by $0 \leq \xi_{1} \leq \xi_{2} \leq \dots \leq \xi_{n} < 2\pi$. 

\paragraph{Proof of Corollary \ref{coro:smooth}.} The asymptotics for the cumulants $\{\kappa_{j}\}_{j=1}^{+\infty}$ follow directly from \eqref{cum of W}, Theorem \ref{thm:MGF} (with $m=0$, $\alpha_{0}=0$ and with $W$ replaced by $tW$), and the fact that \eqref{exp in thm} can be differentiated any number of time with respect to $t$ without worsening the error term (see Remark \ref{remark:diff}). Furthermore, if $W$ is non-constant, then $\sum_{k = 1}^{+\infty} kW_{k}W_{-k} = \sum_{k = 1}^{+\infty} k|W_{k}|^{2} > 0$ (because $W$ is assumed to be real-valued) and from Theorem \ref{thm:MGF} (with $m=0$, $\alpha_{0}=0$ and with $W$ replaced by $\frac{tW}{(2\sum_{k = 1}^{+\infty} kW_{k}W_{-k})^{1/2}}$, $t \in \mathbb{R}$) we also have
\begin{align*}
\mathbb{E}\bigg[ \exp \bigg( t\frac{\sum_{j=1}^{n}W(e^{i\phi_{j}})-n\int_0^{2\pi} W(e^{i\phi})d\mu_V(e^{i\phi})}{(2\sum_{k = 1}^{+\infty} kW_{k}W_{-k})^{1/2}} \bigg) \bigg] = e^{\frac{t^{2}}{2}+ \bigO(n^{-1})},
\end{align*}
as $n \to + \infty$ with $t\in \mathbb{R}$ arbitrary but fixed. The convergence in distribution stated in Corollary \ref{coro:smooth} now follows from standard theorems (see e.g. \cite[top of page 415]{Bill}).

\paragraph{Proof of Corollary \ref{coro:log}.} The proof is similar to the proof of Corollary \ref{coro:smooth}. The main difference is that (i) for the asymptotics of the cumulants, one needs to use Theorem \ref{thm:MGF} with $W=0$, $m=0$ if $t = 1$, and with $W=0$, $m=1$, $\alpha_0 = 0$, $u_{1}=0$ if $t \in \mathbb{T} \setminus \{1\}$, and (ii) for the convergence in distribution, one needs to use Theorem \ref{thm:MGF} with $W=0$, $m=0$, and $\alpha_{0}$ replaced by $\alpha\sqrt{2}/\sqrt{\log n}$, $\alpha \in \mathbb{R}$ fixed, if $t = 1$, and with $W=0$, $m=1$, $\alpha_0 = 0$, $u_{1}=0$ and $\alpha_1$ replaced by $\alpha\sqrt{2}/\sqrt{\log n}$, $\alpha \in \mathbb{R}$ fixed, if $t \in \mathbb{T}\setminus \{1\}$.

\paragraph{Proof of Corollary \ref{coro:counting}.} This proof is also similar to the proof of Corollary \ref{coro:smooth}. For the asymptotics of the cumulants, one needs to use Theorem \ref{thm:MGF} with $W=0$, $m=1$, $\alpha_{0}=\alpha_{1}=0$ and for the convergence in distribution, one needs to use Theorem \ref{thm:MGF} with $W=0$, $m=1$, $\alpha_{0}=\alpha_{1}=0$, and with $u_{1}$ replaced by $\pi u/\sqrt{\log n}$, $u\in \mathbb{R}$ fixed.

\paragraph{Proof of Corollary \ref{coro:order}.} 

The proof is inspired by Gustavsson \cite[Theorem 1.2]{Gustavsson}. Let $\theta\in (0,2\pi)$ and $k_{\theta}=[n \int_{0}^{\theta}d\mu_{V}(e^{i \phi})]$, where $[x]:= \lfloor x + \frac{1}{2}\rfloor$, and consider the random variable
\begin{align}\label{Yn tj def}
Y_{n} := \frac{n\int_{0}^{\xi_{k_{\theta}}}d\mu_{V}(e^{i \phi}) - k_{\theta}}{\sqrt{\log n}/\pi} = \frac{\mu_{n}(\xi_{k_{\theta}})-k_{\theta}}{\sigma_{n}}, 
\end{align} 
where $\mu_{n}(\xi) := n\int_{0}^{\xi}d\mu_{V}(e^{i \phi})$ and $\sigma_{n} := \frac{1}{\pi}\sqrt{\log n}$. 
For $y \in \mathbb{R}$, we have
\begin{align}
& \mathbb{P}\big[ Y_{n} \leq y \big] = \mathbb{P}\Big[\xi_{k_{\theta}} \leq \mu_{n}^{-1}\big(k_{\theta} + y \sigma_{n}\big) \Big] = \mathbb{P}\Big[N_{n}\Big(\mu_{n}^{-1}\big(k_{\theta} + y \sigma_{n} \big)\Big) \geq k_{\theta} \Big]. \label{prob1}
\end{align}
Letting $\tilde{\theta} := \mu_{n}^{-1}\big(k_{\theta} + y \sigma_{n} \big)$, we can rewrite \eqref{prob1} as
\begin{align}\label{PYny}
\mathbb{P}\big[ Y_{n} \leq y \big] & = \mathbb{P}\bigg[ \frac{N_{n}(\tilde{\theta})-\mu_{n}(\tilde{\theta})}{\sqrt{\sigma_{n}^{2}}} \geq \frac{k_{\theta}-\mu_{n}(\tilde{\theta})}{\sigma_{n}} \bigg] = \mathbb{P}\bigg[ \frac{\mu_{n}(\tilde{\theta})-N_{n}(\tilde{\theta})}{\sigma_{n}} \leq y \bigg].
\end{align}
As $n \to +\infty$, we have 
\begin{align}\label{tj tilde remain bounded away from each other}
k_{\theta} = [\mu_{n}(\theta)] = \bigO(n), \qquad \tilde{\theta} = \theta \Big(1+\bigO\Big(\tfrac{\sqrt{\log n}}{n}\Big)\Big). 
\end{align} 
Since Theorem \ref{thm:MGF} also holds in the case where $\theta$ depends on $n$ but remains bounded away from $0$, the same is true for the convergence in distribution in Corollary \ref{coro:counting}. 
By \eqref{tj tilde remain bounded away from each other}, $\tilde{\theta}$ remains bounded away from $0$, and therefore Corollary \ref{coro:counting} together with (\ref{PYny}) implies that $Y_{n}$ converges in distribution to a standard normal random variable. 
Since
\vspace{-0.1cm}
\begin{align*}
\mathbb{P}\bigg[ \frac{n\psi(e^{i\eta_{k_{\theta}}})}{\sqrt{\log n}/\pi}(\xi_{k_{\theta}}-\eta_{k_{\theta}}) \leq y \bigg] 
& = \mathbb{P}\bigg[ Y_{n} \leq  \frac{\mu_{n}(\eta_{k_{\theta}} + y \frac{\sigma_{n}}{n\psi(e^{i\eta_{k_{\theta}}})})-\mu_{n}(\eta_{k_{\theta}})}{\sigma_{n}} \bigg] 
	\\
& = \mathbb{P}\bigg[ Y_{n} \leq \int_{\eta_{k_\theta}}^{\eta_{k_\theta} + \frac{y\sigma_n}{n \psi(e^{i \eta_{k_\theta}})}} \frac{n\psi(e^{i\phi})}{\sigma_n} d\phi \bigg] 
= \mathbb{P}\big[ Y_{n} \leq y +o(1) \big]
\end{align*}
\normalsize
as $n\to + \infty$, this implies the convergence in distribution in the statement of Corollary \ref{coro:order}.

\paragraph{Proof of Corollary \ref{coro:rigidity}.}
Let $\mu_{n}(\xi) := n\int_{0}^{\xi}d\mu_{V}(e^{i \phi})$, $\sigma_{n} := \frac{1}{\pi}\sqrt{\log n}$, and for $\theta \in [0,2\pi)$, let $\overline{N}_{n}(\theta) := N_{n}(\theta)-\mu_{n}(\theta)$. Using Theorem \ref{thm:MGF} with $W=0$, $m \in \N_{>0}$, $\alpha_{0}=\ldots=\alpha_{m}=0$ and $u_{1},\ldots,u_{m} \in \mathbb{R}$, we infer that for any $\delta \in (0,\pi)$ and $M>0$, there exists $n_{0}'=n_{0}'(\delta,M)\in \mathbb{N}$ and $\mathrm{C}=\mathrm{C}(\delta,M)>0$ such that
\begin{align}\label{expmomentbound}
\mathbb{E} \big( e^{\sum_{k=1}^{m}u_{k}\overline{N}_{n}(\theta_{k})} \big) \leq  \mathrm{C} \exp \bigg( \frac{\sum_{k=0}^{m} u_{k}^{2}}{2}\frac{\sigma_{n}^{2}}{2} \bigg),
\end{align}
for all $n\geq n_{0}'$, $(\theta_{1},\ldots, \theta_{m})$ in compact subsets of $(0,2\pi)_{\mathrm{ord}}^{m}\cap (\delta,2\pi-\delta)^{m}$ and $u_{1},\ldots,u_{m} \in [-M,M]$, and where $u_{0}=-u_{1}-\ldots-u_{m}$.

\begin{lemma}\label{lemma: A r eps}
For any $\delta \in (0,\pi)$, there exists $c>0$ such that for all large enough $n$ and small enough $\epsilon>0$,
\begin{align}\label{prob statement lemma 2.1}
\mathbb P\left(\sup_{\delta \leq \theta \leq 2\pi-\delta}\bigg|\frac{N_{n}(\theta)-\mu_{n}(\theta)}{\sigma^2_{n}}\bigg|\leq \pi (1+\epsilon) \right) \geq 1-\frac{c}{\log n}.
\end{align}
\end{lemma}
\begin{proof}
A naive adaptation of \cite[Lemma 8.1]{CharlierMB} (an important difference between \cite{CharlierMB} and our situation is that $\sigma_{n} = \frac{\sqrt{\log n}}{\sqrt{2}\pi}$ in \cite{CharlierMB} while here we have $\smash{\sigma_{n} = \frac{\sqrt{\log n}}{\pi}}$) yields 
\begin{align*}
\mathbb P\left(\sup_{\delta \leq \theta \leq 2\pi-\delta}\bigg|\frac{N_{n}(\theta)-\mu_{n}(\theta)}{\sigma^2_{n}}\bigg|\leq \sqrt{2}\pi (1+\epsilon) \right) \geq 1-o(1).
\end{align*}
The inequality \eqref{expmomentbound} can in fact be used to obtain the stronger statement \eqref{prob statement lemma 2.1}.\footnote{We are very grateful to a referee for pointing this out.} Recall that $\eta_{k}= \mu_{n}^{-1}(k)$ is the classical location of the $k$-th smallest point $\xi_{k}$ and is defined in \eqref{def of kappa k}. Since $\mu_{n}$ and $N_{n}$ are increasing functions, for $x\in[\eta_{k-1},\eta_k]$ with $k \in \{1,\ldots,n\}$, we have
\begin{equation}\label{lol1}
N_{n}(x)-\mu_{n}(x)\leq N_{n}(\eta_k)-\mu_{n}(\eta_{k-1})
=N_{n}(\eta_k)-\mu_{n}(\eta_{k})+1,
\end{equation}
which implies
\begin{align*}
\sup_{\delta \leq x \leq 2\pi-\delta}\frac{N_{n}(x)-\mu_{n}(x)}{\sigma_{n}^2} \leq \sup_{k\in \mathcal{K}_{n}}
\frac{N_{n}(\eta_k)-\mu_{n}(\eta_{k})+1}{\sigma_{n}^2},
\end{align*}
where $\mathcal{K}_{n} = \{k: \eta_{k}>\delta \mbox{ and } \eta_{k-1}<2\pi-\delta\}$. Hence, for any $v>0$,
\begin{align*}
\mathbb P\left(\sup_{\delta \leq x \leq 2\pi-\delta}\frac{N_{n}(x)-\mu_{n}(x)}{\sigma^2_{n}}>v\right)\leq \mathbb P\left( \sup_{k\in \mathcal{K}_{n}} \frac{N_{n}(\kappa_k)-\mu_{n}(\kappa_{k})}{\sigma_{n}^2}>v-\frac{1}{\sigma_{n}^{2}}\right).
\end{align*}
Let $\epsilon_{0}>0$ be small and fixed, and let $\mathcal{I}$ be an arbitrary but fixed subset of $(0,\epsilon_{0}]$. The claim \eqref{prob statement lemma 2.1} will follow if we can prove for any $\epsilon \in \mathcal{I}$ that
\begin{align}\label{lol9}
\mathbb P\left( \sup_{k\in \mathcal{K}_{n}} \frac{N_{n}(\eta_{k})-\mu_{n}(\eta_{k})}{\sigma_{n}^2}>\pi(1+\epsilon)\right) \leq \frac{c_{1}}{\log n},
\end{align}
for some $c_{1}=c_{1}(\mathcal{I})>0$. Let $m \in \N$ be fixed and $S_{m}$ and $S_{m}'$ be the following two collections of points of size $m$
\begin{align*}
& S_{m} = \bigg\{ \delta + (2\pi-2\delta) \frac{4j+1}{4m}: \qquad j=0,\ldots,m-1 \bigg\}, \\
& S_{m}' = \bigg\{ \delta + (2\pi-2\delta) \frac{4j+2}{4m}: \qquad j=0,\ldots,m-1 \bigg\}.
\end{align*}
Let $X_{n}(\theta):= (N_{n}(\theta)-\mu_{n}(\theta))/\sigma_{n}$. For any $\theta \in [\delta,2\pi-\delta]$, we have by Corollary \ref{coro:counting} that $\mathbb{E}[X_{n}(\theta)] = \bigO(\frac{\sqrt{\log n}}{n})$ and $\mbox{Var}[X_{n}(\theta)] \leq 2$ for all large enough $n$. Hence, by Chebyshev's inequality, for any fixed $\ell > 0$, $\mathbb{P}\big(\frac{|X_{n}|}{\sigma_{n}} \geq \ell\big) \leq \frac{3}{\ell^{2}\sigma_{n}^{2}}$ for all large enough $n$. Using this inequality with $\ell=\frac{\pi\epsilon}{2}$ together with a union bound, we get 
\begin{align*}
\mathbb P\left(\sup_{\hat{\theta} \in S_{m}\cup S_{m}'}\bigg|\frac{N_{n}(\hat{\theta})-\mu_{n}(\hat{\theta})}{\sigma^2_{n}}\bigg| > \frac{\pi\epsilon}{2} \right) = \mathbb P\left(\sup_{\hat{\theta} \in S_{m}\cup S_{m}'}\bigg|\frac{X_{n}(\hat{\theta})}{\sigma_{n}}\bigg| > \frac{\pi\epsilon}{2} \right) \leq \frac{3 \times 2m}{(\frac{\pi\epsilon}{2})^{2}\frac{\log n}{\pi^{2}}} = \frac{24m}{\epsilon^{2}\log n},
\end{align*}
and then
\begin{align}
& \PP \left(\sup_{k \in \mathcal{K}_{n}}\bigg|\frac{N_{n}(\eta_{k})-\mu_{n}(\eta_{k})}{\sigma^2_{n}}\bigg| > \pi (1+\epsilon) \right) \leq \frac{24 m}{\epsilon^{2}\log n}  \nonumber \\
& \hspace{1cm} + \sum_{k \in \mathcal{K}_{n}} \PP \bigg(  \bigg|\frac{N_{n}(\eta_{k})-\mu_{n}(\eta_{k})}{\sigma^2_{n}}\bigg| > \pi (1+\epsilon) \quad \mbox{ and } \quad \sup_{\hat{\theta} \in S_{m}\cup S_{m}'}\bigg|\frac{X_{n}(\hat{\theta})}{\sigma_{n}}\bigg| \leq  \frac{\pi \epsilon}{2} \bigg). \label{lol6}
\end{align}
The reason for introducing two subsets $S_{m},S_{m}'$ is the following: for any $k \in \mathcal{K}_{n}$, one must have that $\eta_{k}$ remains bounded away from at least one of $S_{m},S_{m}'$ (so that \eqref{expmomentbound} can be applied). Indeed, suppose for example that $\theta$ is bounded away from $S_{m}$, then by \eqref{expmomentbound} (with $m$ replaced by $m+1$ and with $u_{1}=u$ and $u_{2}=\ldots=u_{m+1}=-\frac{u}{m}$) we have
\begin{align*}
\E \bigg[ \exp \bigg( u \overline{N}_{n}(\eta_{k})-\frac{u}{m}\sum_{\hat{\theta}\in S_{m}}\overline{N}_{n}(\hat{\theta}) \bigg) \bigg] \leq \mathrm{C} \exp \bigg\{ \frac{u^{2}\sigma_{n}^{2}}{4}\bigg( 1+\frac{1}{m} \bigg) \bigg\}
\end{align*}
and similarly, 
\begin{align*}
\E \bigg[ \exp \bigg( \frac{u}{m}\sum_{\hat{\theta}\in S_{m}}\overline{N}_{n}(\hat{\theta}) - u \overline{N}_{n}(\eta_{k}) \bigg) \bigg] \leq \mathrm{C} \exp \bigg\{ \frac{u^{2}\sigma_{n}^{2}}{4}\bigg( 1+\frac{1}{m} \bigg) \bigg\}.
\end{align*}
Hence, if $\eta_{k}$ remains bounded away from $S_{m}$, we have (with $\gamma:=\pi(1+\epsilon/2)$ and $\alpha:= \frac{1}{2}(1+\frac{1}{m})$)
\begin{align}
& \PP \bigg( \bigg|\frac{\overline{N}_{n}(\eta_{k})}{\sigma^2_{n}}\bigg| > \pi (1+\epsilon) \quad \mbox{ and } \quad \sup_{\hat{\theta} \in S_{m}\cup S_{m}'}\bigg| \frac{\overline{N}_{n}(\hat{\theta})}{\sigma_{n}^{2}}\bigg| \leq  \frac{\pi \epsilon}{2} \bigg) \label{lol7} \\
& \leq \PP \bigg(  \frac{\overline{N}_{n}(\eta_{k})-\frac{1}{m}\sum_{\hat{\theta}\in S_{m}}\overline{N}_{n}(\hat{\theta})}{\sigma_{n}^{2}}  > \gamma \bigg) + \PP \bigg(  \frac{\frac{1}{m}\sum_{\hat{\theta}\in S_{m}}\overline{N}_{n}(\hat{\theta})-\overline{N}_{n}(\eta_{k})}{\sigma_{n}^{2}}  > \gamma \bigg) \nonumber \\
&  = \PP \bigg( e^{\frac{\gamma}{\alpha}(\overline{N}_{n}(\eta_{k})-\frac{1}{m}\sum_{\hat{\theta}\in S_{m}}\overline{N}_{n}(\hat{\theta}))}  > e^{\frac{\gamma^{2}}{\alpha}\sigma_{n}^{2}} \bigg) + \PP \bigg( e^{\frac{\gamma}{\alpha}(\frac{1}{m}\sum_{\hat{\theta}\in S_{m}}\overline{N}_{n}(\hat{\theta})-\overline{N}_{n}(\eta_{k}))}  > e^{\frac{\gamma^{2}}{\alpha}\sigma_{n}^{2}} \bigg) \nonumber \\
& \leq \E \bigg( e^{\frac{\gamma}{\alpha}(\overline{N}_{n}(\eta_{k})-\frac{1}{m}\sum_{\hat{\theta}\in S_{m}}\overline{N}_{n}(\hat{\theta}))}  \bigg)e^{-\frac{\gamma^{2}}{\alpha} \sigma_{n}^{2}} + \E \bigg( e^{\frac{\gamma}{\alpha}(\frac{1}{m}\sum_{\hat{\theta}\in S_{m}}\overline{N}_{n}(\hat{\theta})-\overline{N}_{n}(\eta_{k}))} \bigg)e^{-\frac{\gamma^{2}}{\alpha}\sigma_{n}^{2}} \nonumber \\
& \leq 2 \mathrm{C} \exp \bigg( \frac{\gamma^{2}\sigma_{n}^{2}}{4\alpha^{2}}\bigg( 1+\frac{1}{m} \bigg)-\frac{\gamma^{2}}{\alpha}\sigma_{n}^{2} \bigg) = 2 \mathrm{C} \exp \bigg( - \frac{\pi^{2}(1+\frac{\epsilon}{2})^{2}\sigma_{n}^{2}}{1+\frac{1}{m}} \bigg) = 2 \mathrm{C} n^{- \frac{(1+\frac{\epsilon}{2})^{2}}{1+\frac{1}{m}}}. \label{lol8}
\end{align}
We obtain the same bound \eqref{lol8} if $\eta_{k}$ in \eqref{lol7} is instead bounded away from $S_{m}'$. The above exponent is less than $-1$ provided that $m$ is sufficiently large relative to $\epsilon$. Since the number of points in $\mathcal{K}_{n}$ is proportional to $n$, the claim \eqref{prob statement lemma 2.1} now directly follows from \eqref{lol6} (recall also \eqref{lol9}). 
\end{proof}

\begin{lemma}\label{lemma: xk not far away from kappa k}
Let $\delta \in (0,\frac{\pi}{2})$ and $\epsilon > 0$. For all sufficiently large $n$, if the event
\begin{align}\label{event holds true}
\sup_{\delta \leq \theta \leq 2\pi-\delta}\left|\frac{N_{n}(\theta)-\mu_{n}(\theta)}{\sigma^2_{n}}\right| \leq \pi(1+\epsilon)
\end{align}
holds true, then we have
\begin{align}\label{upper and lower bound for mu xk}
\sup_{k \in (\mu_{n}(2\delta),\mu_{n}(2\pi-2\delta))} \bigg|\frac{\mu_{n}(\xi_k) - k}{\sigma^2_{n}}\bigg| \leq \pi (1+\epsilon) + \frac{1}{\sigma_{n}^{2}},
\end{align}
\end{lemma}
\begin{proof}
The proof is almost identical to the proof of \cite[Lemma 8.2]{CharlierMB} so we omit it.
\end{proof}

By combining Lemmas \ref{lemma: A r eps} and \ref{lemma: xk not far away from kappa k}, we arrive at the following result (the proof is very similar to \cite[Proof of (1.38)]{CharlierMB}, so we omit it).
\begin{lemma}\label{lemma:rigidity ordered}
For any $\delta \in (0,\pi)$, there exists $c>0$ such that for all large enough $n$ and small enough $\epsilon>0$,
\begin{align}
& \mathbb{P}\bigg( \max_{\delta n \leq k \leq (1-\delta)n}  \psi(e^{i\eta_{k}})|\xi_{k}-\eta_{k}| \leq \frac{1+\epsilon}{\pi} \frac{\log n}{n} \bigg) \geq 1-\frac{c}{\log n}. \label{lol3}
\end{align}
\end{lemma}

\paragraph{Extending Lemmas \ref{lemma: A r eps} and \ref{lemma:rigidity ordered} to $\delta=0$.} In this paper, the support of $\mu_{V}$ is $\mathbb{T}$. Therefore, the point $1 \in \mathbb{T}$ should play no special role in the study of the global rigidity of the points, which suggests that \eqref{prob statement lemma 2.1} and \eqref{lol3} should still hold with $\delta=0$. The next lemma shows that this is indeed the case.

\begin{lemma}\label{lemma:delta=0}(Proof of \eqref{probabilistic upper bound 1}.)
For each small enough $\epsilon >0$, there exists $c>0$ such that
\begin{align*}
\mathbb P\left(\sup_{0 \leq \theta < 2\pi}\bigg|\frac{N_{n}(\theta)-\mu_{n}(\theta)}{\sigma^2_{n}}\bigg|\leq \pi(1+\epsilon) \right) \geq 1-\frac{c}{\log n}
\end{align*}
for all large enough $n$.
\end{lemma}
\begin{proof}
For $-\pi \leq \theta < 0$, let $\tilde{N}_{n}(\theta):=\#\{\phi_{j}-2\pi \in (-\pi,\theta]\}$, and for $0 \leq \theta < \pi$, let $\tilde{N}_{n}(\theta):=\#(\{\phi_{j}-2\pi \in (-\pi,0]\}\cup \{\phi_{j} \in [0,\theta]\})$. For $-\pi \leq \theta < \pi$, define also $\tilde{\mu}_{n}(\theta):=n\int_{-\pi}^{\theta}d\mu_{V}(e^{i\phi})$. In the same way as for Lemma \ref{lemma: A r eps}, the following holds: for any $\delta \in (0,\pi)$, there exists $c_{1}>0$ such that for all large enough $n$ and small enough $\epsilon>0$,
\begin{align*}
\mathbb P\left(\sup_{-\pi+\delta \leq \theta \leq \pi-\delta}\bigg|\frac{\tilde{N}_{n}(\theta)-\tilde{\mu}_{n}(\theta)}{\sigma^2_{n}}\bigg|\leq \pi(1+\epsilon) \right) \geq 1-\frac{c_{1}}{\log n}.
\end{align*}
Clearly, 
\begin{align*}
\tilde{N}_{n}(\theta) = \begin{cases}
N_{n}(\theta+2\pi)-N_{n}(\pi), & \mbox{if } \theta \in (-\pi,0), \\
N_{n}(\theta)+n-N_{n}(\pi), & \mbox{if } \theta \in (0,\pi),
\end{cases} \quad \tilde{\mu}_{n}(\theta) = \begin{cases}
\mu_{n}(\theta+2\pi)-\mu_{n}(\pi), & \mbox{if } \theta \in (-\pi,0), \\
\mu_{n}(\theta)+n-\mu_{n}(\pi), & \mbox{if } \theta \in (0,\pi),
\end{cases}
\end{align*}
and therefore 
\begin{align*}
\frac{\tilde{N}_{n}(\theta)-\tilde{\mu}_{n}(\theta)}{\sigma^2_{n}} = - \frac{N_{n}(\pi)-\mu_{n}(\pi)}{\sigma_{n}^{2}} + \begin{cases}
\frac{N_{n}(\theta+2\pi)-\mu_{n}(\theta+2\pi)}{\sigma^2_{n}}, & \mbox{if } \theta \in (-\pi,0), \\[0.1cm]
\frac{N_{n}(\theta)-\mu_{n}(\theta)}{\sigma^2_{n}}, & \mbox{if } \theta \in (0,\pi).
\end{cases}
\end{align*}
Thus, for all large enough $n$,
\begin{align*}
\mathbb P\left(\sup_{\theta \in [0,2\pi)\setminus (\pi-\delta,\pi+\delta)}\bigg|\frac{N_{n}(\theta)-\mu_{n}(\theta)}{\sigma^2_{n}}\bigg|\leq \pi(1+\epsilon) + \bigg|\frac{N_{n}(\pi)-\mu_{n}(\pi)}{\sigma_{n}^{2}} \bigg| \right) \geq 1-\frac{c_{1}}{\log n}.
\end{align*}
Combining the above with \eqref{prob statement lemma 2.1} (with $c$ replaced by $c_{2}$), we obtain
\begin{align}\label{lol4}
\mathbb P\left(\sup_{\theta \in [0,2\pi)}\bigg|\frac{N_{n}(\theta)-\mu_{n}(\theta)}{\sigma^2_{n}}\bigg|\leq \pi(1+\epsilon) + \bigg|\frac{N_{n}(\pi)-\mu_{n}(\pi)}{\sigma_{n}^{2}} \bigg| \right) \geq 1- \frac{c_{1}+c_{2}}{\log n}.
\end{align}
Let $X_{n}:= (N_{n}(\pi)-\mu_{n}(\pi))/\sigma_{n}$. By Corollary \ref{coro:counting}, $\mathbb{E}[X_{n}] = \bigO(\frac{\sqrt{\log n}}{n})$ and $\mbox{Var}[X_{n}] \leq 2$ for all large enough $n$. Hence, by Chebyshev's inequality, for any fixed $\ell > 0$, $\mathbb{P}\big(\frac{|X_{n}|}{\sigma_{n}} \geq \ell\big) \leq \frac{3}{\ell^{2}\sigma_{n}^{2}}$ for all large enough $n$. Applying this inequality with $\ell= \pi(1+(1+\epsilon)\epsilon) - \pi (1+\epsilon) = \pi \epsilon^{2}$ we see that if $\mathbb{P}(A)$ denotes the left-hand side of \eqref{lol4}, then
\begin{align*}
\mathbb{P}(A) 
\leq \mathbb{P}\Big(A \cap \big\{\tfrac{|X_{n}|}{\sigma_{n}} < \ell\big\}\Big) + \mathbb{P}\Big(\tfrac{|X_{n}|}{\sigma_{n}} \geq \ell\Big)
\leq  \mathbb{P}\Big(A \cap \big\{\tfrac{|X_{n}|}{\sigma_{n}} < \ell\big\}\Big) + \frac{3}{\ell^{2}\sigma_{n}^{2}}.
\end{align*}
Together with \eqref{lol4}, this gives
\begin{align*}
\mathbb P\bigg(\sup_{\theta \in [0,2\pi)}\bigg|\frac{N_{n}(\theta)-\mu_{n}(\theta)}{\sigma^2_{n}}\bigg|\leq \pi\big(1+(1+\epsilon)\epsilon\big) \bigg) 
& \geq \mathbb{P}\Big(A \cap \big\{\tfrac{|X_{n}|}{\sigma_{n}} < \ell\big\}\Big)
\geq \mathbb{P}(A) - \frac{3}{\ell^{2}\sigma_{n}^{2}}
	\\
& \geq 1- \frac{c_{1}+c_{2}}{\log n} - \frac{3}{\ell^{2}\sigma_{n}^{2}}
\geq 1-\frac{c_{3}}{\log n},
\end{align*}
for some $c_{3}=c_{3}(\epsilon)>0$, which proves the claim.
\end{proof}
The upper bound \eqref{probabilistic upper bound 2} can be proved using the same idea as in the proof of Lemma \ref{lemma:delta=0}.

\appendix 

\section{Equilibrium measure}\label{section:equilibrium measure}
Assume that $\mu_{V}$ is supported on $\mathbb{T}$. We make the ansatz that $\mu_{V}$ is of the form \eqref{uni: dmu} for some $\psi$. Let $g$ be as in \eqref{g-def}. Substituting \eqref{g+ + g-} in \eqref{var equality} and differentiating, we obtain
\begin{align*}
g_{+}'(z)+g_{-}'(z) = V'(z) + \frac{1}{z}, \qquad z \in \mathbb{T}.
\end{align*}
Since $g'(z) = \frac{1}{z}+\bigO(z^{-2})$ as $z \to \infty$, we deduce that
\begin{align}\label{gprime}
g'(z) = -\frac{\varphi(z)}{2\pi i}\int_{\mathbb{T}}\frac{\frac{1}{s}+V'(s)}{s-z}ds, \qquad z \in \mathbb{C}\setminus \mathbb{T},
\end{align}
where $\varphi(z) := +1$ if $|z|>1$ and $\varphi(z) := -1$ if $|z|<1$. Using \eqref{gprime} in \eqref{uni: gprime}, it follows that
\begin{align}\label{lol1}
 -\frac{2\pi}{z}\psi(z) = \frac{1}{\pi i}\dashint_{\mathbb{T}} \frac{\frac{1}{s}+V'(s)}{s-z}ds, \qquad z \in \mathbb{T}.
\end{align}
Recall from \eqref{V as a Laurent series} that $V$ is analytic in the open annulus $U$ and real-valued on $\mathbb{T}$, and therefore
\begin{align*}
V(z) = V_{0} + \sum_{k \geq 1} (V_{k}z^{k}+\overline{V_{k}}z^{-k}), \qquad V'(z) = \sum_{k \geq 1} (k V_{k}z^{k-1}- k\overline{V_{k}}z^{-k-1}), \qquad z \in U.
\end{align*}
(It is straightforward to check that the series $\sum_{k \geq 1} k V_{k}z^{k-1}$ and $\sum_{k \geq 1} k\overline{V_{k}}z^{-k-1}$ are convergent in $U$.) Direct computation gives
\begin{align*}
\dashint_{\mathbb{T}} \frac{\frac{1}{s}+V'(s)}{s-z}ds = - \frac{\pi i}{z} + \pi i \sum_{k \geq 1} (k V_{k}z^{k-1}+ k\overline{V_{k}}z^{-k-1}), \qquad z \in \mathbb{T},
\end{align*}
which, by \eqref{lol1}, proves that $\psi$ is given by \eqref{psi solved explicitly}. Since the right-hand side of \eqref{psi solved explicitly} is positive on $\mathbb{T}$ (by our assumption that $V$ is regular), we conclude that $\psi(e^{i\theta})d\theta$ is a probability measure satisfying the Euler--Lagrange condition \eqref{var equality}. Therefore $\psi(e^{i\theta})d\theta$ minimizes \eqref{functional}, i.e. $\psi(e^{i\theta})d\theta$ is the equilibrium measure associated to $V$. Since the equilibrium measure is unique \cite{SaTo}, this proves \eqref{uni: dmu}.

\section{Confluent hypergeometric model RH problem}\label{appendix:hypergeometric RHP}
	
\begin{itemize}
	\item[(a)] $\Phi_{\mathrm{HG}} : \mathbb{C} \setminus \Sigma_{\mathrm{HG}} \rightarrow \mathbb{C}^{2 \times 2}$ is analytic, where $\Sigma_{\mathrm{HG}}$ is shown in Figure \ref{Fig:HG}.
	\item[(b)] For $z \in \Gamma_{k}$ (see Figure \ref{Fig:HG}), $k = 1,\dots,8$, $\Phi_{\mathrm{HG}}$ obeys the jump relations
	\begin{equation}\label{jumps PHG3}
		\Phi_{\mathrm{HG},+}(z) = \Phi_{\mathrm{HG},-}(z)J_{k},
	\end{equation}
	where
	\begin{align*}
		& J_{1} = \begin{pmatrix}
		0 & e^{-i\pi \beta} \\ -e^{i\pi\beta} & 0
		\end{pmatrix}, \quad J_{5} = \begin{pmatrix}
		0 & e^{i\pi\beta} \\ -e^{-i\pi\beta} & 0
		\end{pmatrix},\quad J_{3} = J_{7} = \begin{pmatrix}
		e^{\frac{i\pi\alpha}{2}} & 0 \\ 0 & e^{-\frac{i\pi\alpha}{2}}
		\end{pmatrix}, \\
		& J_{2} = \begin{pmatrix}
		1 & 0 \\ e^{-i\pi\alpha}e^{i\pi\beta} & 1
		\end{pmatrix}\hspace{-0.1cm}, \hspace{-0.3cm} \quad J_{4} = \begin{pmatrix}
		1 & 0 \\ e^{i\pi\alpha}e^{-i\pi\beta} & 1
		\end{pmatrix}\hspace{-0.1cm}, \hspace{-0.3cm} \quad J_{6} = \begin{pmatrix}
		1 & 0 \\ e^{-i\pi\alpha}e^{-i\pi\beta} & 1
		\end{pmatrix}\hspace{-0.1cm}, \hspace{-0.3cm} \quad J_{8} = \begin{pmatrix}
		1 & 0 \\ e^{i\pi\alpha}e^{i\pi\beta} & 1
		\end{pmatrix}.
	\end{align*}
\item[(c)] As $z \to \infty$, $z \notin \Sigma_{\mathrm{HG}}$, we have
\begin{equation}\label{Asymptotics HG}
	\Phi_{\mathrm{HG}}(z) = \left( I + \sum_{k=1}^{\infty} \frac{\Phi_{\mathrm{HG},k}}{z^{k}} \right) z^{-\beta\sigma_{3}}e^{-\frac{z}{2}\sigma_{3}}M^{-1}(z),
\end{equation}
where 
\begin{equation}\label{def of tau}
	\Phi_{\mathrm{HG},1} = \Big(\beta^{2}-\frac{\alpha^{2}}{4}\Big) \begin{pmatrix}
	-1 & \tau(\alpha,\beta) \\ - \tau(\alpha,-\beta) & 1
	\end{pmatrix}, \qquad \tau(\alpha,\beta) = -\frac{\Gamma\left( \frac{\alpha}{2}-\beta \right)}{\Gamma\left( \frac{\alpha}{2}+\beta + 1 \right)},
\end{equation}
and
\begin{equation}
	M(z) = \left\{ \begin{array}{l l}
	\displaystyle e^{\frac{i\pi\alpha}{4} \sigma_{3}}e^{- i\pi\beta  \sigma_{3}}, & \displaystyle \frac{\pi}{2} < \arg z < \pi, \\
	\displaystyle e^{-\frac{i\pi\alpha}{4} \sigma_{3}}e^{-i\pi\beta  \sigma_{3}}, & \displaystyle \pi < \arg z < \frac{3\pi}{2}, \\
	e^{\frac{i\pi\alpha}{4}\sigma_{3}} \begin{pmatrix}
	0 & 1 \\ -1 & 0
	\end{pmatrix}, & -\frac{\pi}{2} < \arg z < 0, \\
	e^{-\frac{i\pi\alpha}{4}\sigma_{3}} \begin{pmatrix}
	0 & 1 \\ -1 & 0
	\end{pmatrix}, & 0 < \arg z < \frac{\pi}{2}.
	\end{array} \right.
\end{equation}
In \eqref{Asymptotics HG}, $z^{-\beta}$ has a cut along $i\mathbb{R}^{-}$ so that $z^{-\beta} = |z|^{-\beta}e^{-\beta i \arg(z)}$ with $-\frac{\pi}{2} < \arg z < \frac{3\pi}{2}$. As $z \to 0$, we have
\begin{equation}\label{lol 35}
\begin{array}{l l}
	\displaystyle \Phi_{\mathrm{HG}}(z) = \left\{ \begin{array}{l l}
	\begin{pmatrix}
	\bigO(1) & \bigO(\log z) \\
	\bigO(1) & \bigO(\log z)
	\end{pmatrix}, & \mbox{if } z \in II \cup III \cup VI \cup VII, \\
	\begin{pmatrix}
	\bigO(\log z) & \bigO(\log z) \\
	\bigO(\log z) & \bigO(\log z)
	\end{pmatrix}, & \mbox{if } z \in I\cup IV \cup V \cup VIII,
    \end{array} \right., & \displaystyle \mbox{ if } \re \alpha = 0, \\[0.9cm]
		
	\displaystyle \Phi_{\mathrm{HG}}(z) = \left\{ \begin{array}{l l}
	\begin{pmatrix}
	\bigO(z^{\frac{\alpha}{2}}) & \bigO(z^{-\frac{\alpha}{2}}) \\
	\bigO(z^{\frac{\alpha}{2}}) & \bigO(z^{-\frac{\alpha}{2}})
	\end{pmatrix}, & \mbox{if } z \in II \cup III \cup VI \cup VII, \\
	\begin{pmatrix}
	\bigO(z^{-\frac{\alpha}{2}}) & \bigO(z^{-\frac{\alpha}{2}}) \\
	\bigO(z^{-\frac{\alpha}{2}}) & \bigO(z^{-\frac{\alpha}{2}})
	\end{pmatrix}, & \mbox{if } z \in I\cup IV \cup V \cup VIII,
	\end{array} \right. , & \displaystyle \mbox{ if } \re \alpha > 0, \\[0.9cm]
		
	\displaystyle \Phi_{\mathrm{HG}}(z) = \begin{pmatrix}
	\bigO(z^{\frac{\alpha}{2}}) & \bigO(z^{\frac{\alpha}{2}}) \\
	\bigO(z^{\frac{\alpha}{2}}) & \bigO(z^{\frac{\alpha}{2}}) 
	\end{pmatrix}, & \displaystyle \mbox{ if } \re \alpha < 0.
\end{array}
\end{equation}
\end{itemize}
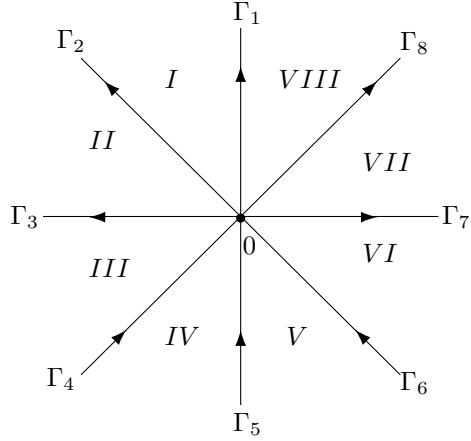
\begin{figure}[t!]
\begin{center}
\setlength{\unitlength}{1truemm}
\begin{picture}(100,55)(-5,10)
\put(50,40){\line(-1,0){26}}
\put(50,40){\line(1,0){26}}        
\put(50,39.8){\thicklines\circle*{1.2}}
\put(50,40){\line(-0.5,0.5){21}}
\put(50,40){\line(-0.5,-0.5){21}}
\put(50,40){\line(0.5,0.5){21}}
\put(50,40){\line(0.5,-0.5){21}}
\put(50,40){\line(0,1){25}}
\put(50,40){\line(0,-1){25}}
\put(50.3,35){$0$}
\put(76.5,39){$\Gamma_7$}        
\put(71,62){$\Gamma_8$}        
\put(49,66){$\Gamma_1$}        
\put(25.8,62.3){$\Gamma_2$}        
\put(19.7,39){$\Gamma_3$}        
\put(24.5,17.5){$\Gamma_4$}
\put(49,11.5){$\Gamma_5$}
\put(71,17){$\Gamma_6$}        
\put(30,39.9){\thicklines\vector(-1,0){.0001}}
\put(68,39.9){\thicklines\vector(1,0){.0001}}
\put(32,58){\thicklines\vector(-0.5,0.5){.0001}}
\put(35,25){\thicklines\vector(0.5,0.5){.0001}}
\put(68,58){\thicklines\vector(0.5,0.5){.0001}}
\put(65,25){\thicklines\vector(-0.5,0.5){.0001}}
\put(50,60){\thicklines\vector(0,1){.0001}}
\put(50,25){\thicklines\vector(0,1){.0001}}
\put(40,57){$I$}
\put(30,49){$II$}
\put(30,32){$III$}
\put(40,23){$IV$}
\put(56,23){$V$}
\put(66,34){$VI$}
\put(66,46){$VII$}
\put(55,57){$VIII$}
\end{picture}
\caption{\label{Fig:HG}The jump contour $\Sigma_{\mathrm{HG}}$ for $\Phi_{\mathrm{HG}}(z)$. The ray $\Gamma_{k}$ is oriented from $0$ to $\infty$, and forms an angle with $\mathbb{R}^{+}$ which is a multiple of $\frac{\pi}{4}$.}
\end{center}
\end{figure}
This model RH problem was first introduced and solved explicitly in \cite{ItsKrasovsky} for the case $\alpha = 0$, and then in \cite{FouMarSou, DIK2011} for the general case. The constant matrices $\Phi_{\mathrm{HG},k}$ depend analytically on $\alpha$ and $\beta$ (they can be found explicitly, see e.g. \cite[eq. (56)]{FouMarSou}). Consider the matrix
\begin{equation}\label{phi_HG}
	\widehat{\Phi}_{\mathrm{HG}}(z) = \begin{pmatrix}
	\frac{\Gamma(1 + \frac{\alpha}{2}-\beta)}{\Gamma(1+\alpha)}G(\frac{\alpha}{2}+\beta, \alpha; z)e^{-\frac{i\pi\alpha}{2}} & -\frac{\Gamma(1 + \frac{\alpha}{2}-\beta)}{\Gamma(\frac{\alpha}{2}+\beta)}H(1+\frac{\alpha}{2}-\beta,\alpha;ze^{-i\pi }) \\
	\frac{\Gamma(1 + \frac{\alpha}{2}+\beta)}{\Gamma(1+\alpha)}G(1+\frac{\alpha}{2}+\beta,\alpha;z)e^{-\frac{i\pi\alpha}{2}} & H(\frac{\alpha}{2}-\beta,\alpha;ze^{-i\pi })
	\end{pmatrix} e^{-\frac{i\pi\alpha}{4}\sigma_{3}},
\end{equation}
where $G$ and $H$ are related to the Whittaker functions:
\begin{equation}\label{relation between G and H and Whittaker}
	G(a,\alpha;z) = \frac{M_{\kappa,\mu}(z)}{\sqrt{z}}, \quad H(a,\alpha;z) = \frac{W_{\kappa,\mu}(z)}{\sqrt{z}}, \quad \mu = \frac{\alpha}{2}, \quad \kappa = \frac{1}{2}+\frac{\alpha}{2}-a.
\end{equation}
The solution $\Phi_{\mathrm{HG}}$ is given by
\begin{equation}
	\Phi_{\mathrm{HG}}(z) = \left\{ \begin{array}{l l}
	\widehat{\Phi}_{\mathrm{HG}}(z)J_{2}^{-1}, & \mbox{ for } z \in I, \\
	\widehat{\Phi}_{\mathrm{HG}}(z), & \mbox{ for } z \in II, \\
	\widehat{\Phi}_{\mathrm{HG}}(z)J_{3}, & \mbox{ for } z \in III, \\
	\widehat{\Phi}_{\mathrm{HG}}(z)J_{3}J_{4}^{-1}, & \mbox{ for } z \in IV, \\
	\widehat{\Phi}_{\mathrm{HG}}(z)J_{2}^{-1}J_{1}^{-1}J_{8}^{-1}J_{7}^{-1}J_{6}, & \mbox{ for } z \in V, \\
	\widehat{\Phi}_{\mathrm{HG}}(z)J_{2}^{-1}J_{1}^{-1}J_{8}^{-1}J_{7}^{-1}, & \mbox{ for } z \in VI, \\
	\widehat{\Phi}_{\mathrm{HG}}(z)J_{2}^{-1}J_{1}^{-1}J_{8}^{-1}, & \mbox{ for } z \in VII, \\
	\widehat{\Phi}_{\mathrm{HG}}(z)J_{2}^{-1}J_{1}^{-1}, & \mbox{ for } z \in VIII. \\
	\end{array} \right.
\end{equation}

\section*{Acknowledgements}

The work of all three authors was supported by the European Research Council, Grant Agreement No. 682537. CC also acknowledges support from the Swedish Research Council, Grant No. 2021-04626. JL also acknowledges support from the Swedish Research Council, Grant No. 2021-03877, and the Ruth and Nils-Erik Stenb\"ack Foundation. We are very grateful to the referees for valuable suggestions, and in particular for providing us with a proof of \eqref{prob statement lemma 2.1}. 


\end{document}